\documentclass[12pt]{amsart}
\usepackage{amssymb,amsmath,graphicx,multicol}
\usepackage{bbm,
amsfonts,amsthm,url,array,enumitem,color}
\usepackage{parskip}
\usepackage{hyperref}

\newtheorem{thm}{Theorem}[section]
\newtheorem{theorem}[thm]{Theorem}

\newtheorem{corollary}[thm]{Corollary}
\newtheorem{lem}[thm]{Lemma}
\newtheorem{lemma}[thm]{Lemma}
\newtheorem{prop}[thm]{Proposition}
\newtheorem{proposition}[thm]{Proposition}
\theoremstyle{definition}
\newtheorem{defn}[thm]{Definition}
\newtheorem{example}[thm]{Example}
\newtheorem{definition}[thm]{Definition}
\newtheorem{rem}[thm]{Remark}  
\newtheorem{remark}[thm]{Remark}  

\DeclareMathOperator{\type}{type}

\DeclareMathOperator{\const}{const}

\newcommand{\DD}{\mathcal D}
\newcommand{\EE}{\mathcal E}
\newcommand{\AAA}{\mathcal A}
\newcommand{\WW}{\mathcal W}
\newcommand{\VV}{\mathcal V}

\newcommand{\zz}{\mathbf z}

\newcommand{\ttt}{\tau}

\newcommand{\indicator}{\mathbbm{1}}
\newcommand{\one}{\mathbbm{1}}

\newcommand{\N}{\mathbb N}

\title[Combinatorics of the two-species ASEP]{Combinatorics of 
the two-species ASEP
and Koornwinder moments}
\date{\today}
\author{Sylvie Corteel}
\address{Laboratoire d'Informatique Algorithmique: Fondements et Applications,
Centre National de la Recherche Scientifique et Universit\'e Paris Diderot,
Paris 7, Case 7014, 75205 Paris Cedex 13
France}
\email{corteel@liafa.univ-paris-diderot.fr}
\author{Olya Mandelshtam}
\address{Department of Mathematics,
University of California, Berkeley, CA}
\email{olya@math.berkeley.edu}
\author{Lauren Williams}
\address{Department of Mathematics,
University of California, Berkeley, CA}
\email{williams@math.berkeley.edu}
\thanks{SC was partially funded by the ``Combinatoire \`a Paris" projet
Emergences 2013--2017 and by ``ALEA Sorbonne" projet IDEX USPC. OM was partially supported by NSF grant DMS-1704874. LW was partially supported by
a Rose-Hills Investigator award and an NSF CAREER award.
All three authors are grateful for the support of the France-Berkeley fund.}

\begin{document}

\begin{abstract}
In previous work \cite{CW-PNAS, CW-Duke1, CW-Duke2}, 
the first and third authors introduced 
staircase tableaux, which they used to give combinatorial formulas
for the stationary distribution of the 
asymmetric simple exclusion process (ASEP) and for the moments
of the Askey-Wilson weight function.  The fact that the
ASEP and Askey-Wilson moments are related at all is quite surprising,
and is due to \cite{USW} (see also \cite{CSSW} for a strengthening
of the relationship).  The ASEP is a model of particles hopping
on a one-dimensional lattice of $N$ sites with open boundaries; 
particles can enter and exit at both left and right borders.  It 
was introduced around 1970 \cite{bio, Spitzer} and is cited
as a model for both traffic flow and translation in protein synthesis.
Meanwhile, the Askey-Wilson polynomials are a famous family
of orthogonal polynomials in one variable; they sit at the top of 
the hierarchy of classical orthogonal polynomials.
So from this previous work, we have the relationship
$$\text{ ASEP } -- \text{ Askey-Wilson moments } -- \text{ staircase tableaux}.$$

\noindent It is well-known that Askey-Wilson polynomials can be viewed as the one-variable
case of the multivariate Koornwinder polynomials, which are also known as
the Macdonald polynomials for the non-reduced type BC root system.  
It is natural then to 
ask whether one can generalize the relationships among
the ASEP, Askey-Wilson moments, and staircase tableaux, in such a way that 
Koornwinder moments replace Askey-Wilson moments.
In \cite{CW-Koornwinder}, we showed that the relationship between Koornwinder moments
and the \emph{two-species ASEP} (a particle model involving two species of particles
with different ``weights'') is parallel to that between Askey-Wilson moments
and the ASEP.  In this article we introduce  \emph{rhombic staircase tableaux}, 
and show that we have the relationship
$$\text{ 2-species ASEP } -- \text{ Koornwinder moments } -- \text{ rhombic staircase tableaux}.$$
In particular, we give 
formulas for the steady state distribution of the two-species
ASEP and for  Koornwinder moments, in terms of 
rhombic staircase tableaux.   
\end{abstract}

\maketitle
\begin{minipage}{\textwidth}
\begin{minipage}{0.25\textwidth}
\centering
\includegraphics[width=1.6in]{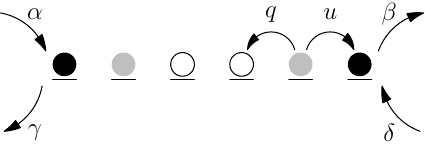}
\end{minipage}
\hspace{.3in}
\begin{minipage}{0.33\textwidth}
 \begin{center}
 \emph{Tableaux formulae\\ Describe hopping particles\\ Please excuse our proofs}
    \end{center}
\end{minipage}
\hspace{.3in}
\begin{minipage}{0.2\textwidth}
\includegraphics[height=1.2in]{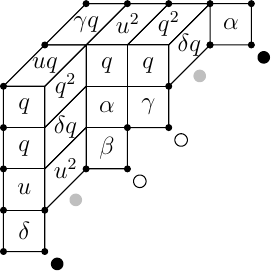}
\end{minipage}
\end{minipage}



\newpage

\setcounter{tocdepth}{1}
\tableofcontents

\section{Introduction}

Introduced in the late 1960's \cite{bio, Spitzer},
the asymmetric exclusion process (ASEP) is a model
of interacting
particles hopping left and right on a one-dimensional lattice of $N$
sites.
In the most general form of the ASEP with open boundaries,
particles may enter and exit at the left with probabilities
$\alpha$ and $\gamma$, and they may exit and enter at the right
with probabilities $\beta$ and $\delta$.  In the bulk,
the probability of hopping left is $q$ times the probability of hopping
right.
Since its introduction, there has been a huge amount of activity on the 
ASEP and its variants for a number of reasons:
for example, it exhibits
boundary-induced phase transitions, it is connected 
to orthogonal polynomials \cite{Sasamoto, USW} and the 
XXZ model \cite{Sandow, ER}, and it is regarded
as a primitive model for translation in protein synthesis \cite{bio}, 
the nuclear pore complex \cite{NPC1, NPC2}, traffic flow \cite{traffic},
and the formation of shocks \cite{shock}.
Moreover, dating back to at least 1982 \cite{SZ}, it was realized
that there were connections between the ASEP and combinatorics.  A main goal 
was to find a combinatorial description of the stationary distribution -- that 
is, to express each component of the stationary distribution as a generating
function for a set of combinatorial objects.
Our work \cite{CW-PNAS, CW-Duke1, CW-Duke2} introduced and used \emph{staircase tableaux}
to give a complete answer to this question, 
building on previous partial results of  
\cite{SZ,  
BE, jumping, Angel, BCEPR, Corteel, CW1, CW2, Viennot}.

A second reason that combinatorialists became intrigued by the ASEP was its link
to orthogonal polynomials \cite{Sasamoto}, and in particular to Askey-Wilson polynomials \cite{USW},
a family
of orthogonal polynomials which are at the top of the hierarchy of
classical orthogonal polynomials in one variable.  
Starting in the early 1980's, 
mathematicians including 
Flajolet \cite{Flajolet}, Viennot \cite{Viennot-book}, and
Foata \cite{Foata},
initiated a combinatorial
approach to orthogonal polynomials, which led to the discovery of 
various combinatorial formulas
for the moments of (the weight functions
of) many of the polynomials in the Askey scheme,
including
$q$-Hermite, Chebyshev, $q$-Laguerre, Charlier, Meixner, and Al-Salam-Chihara
polynomials, see e.g. \cite{IS, ISV, KSZ1,  MSW, SS}.
By combining our tableaux formulas for the stationary distribution of the ASEP with the known link of the 
ASEP and Askey-Wilson moments, we gave the first 
combinatorial formula for 
Askey-Wilson moments in 
\cite{CW-PNAS, CW-Duke1, CW-Duke2}.
In summary, we have the relationships of Figure \ref{AW_triangle}.
\begin{figure}[h]
\centering
\hspace{.9in} \includegraphics[height=.7in]{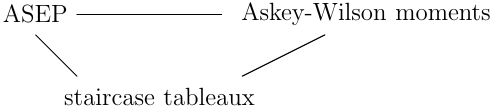}
\caption{}
\label{AW_triangle}
\end{figure}

%


It is well-known that Askey-Wilson polynomials can be viewed as a specialization of
the multivariate \emph{Macdonald-Koornwinder polynomials}, which are
also known as \emph{Koornwinder polynomials} \cite{Koornwinder},
or \emph{Macdonald polynomials for the type BC root system} \cite{Macdonald}.
(For brevity, we will henceforth call them Koornwinder polynomials.)
These polynomials
are particularly important because
the Macdonald polynomials associated to any classical root system
can be expressed as limits or special cases of Koornwinder
polynomials
\cite{vanDiejen}.   Since Koornwinder polynomials generalize
Askey-Wilson polynomials, it is natural to ask if one can 
find an analogue of Figure \ref{AW_triangle} with Koornwinder moments replacing Askey-Wilson moments.

In \cite{CW-Koornwinder}, we showed that there is a relationship between the
\emph{two-species ASEP} and 
\emph{homogeneous Koornwinder moments}  which is parallel
to that between the ASEP and Askey-Wilson moments.
The \emph{two-species ASEP} is a generalization of the ASEP which involves two different
kinds of particles, \emph{heavy} and \emph{light}.  Both types of particles can hop left and right
in a one-dimensional lattice of $N$ sites (heavy and light particles interact exactly as do particles and holes
in the usual ASEP), but only the heavy particles can enter and exit the lattice
at the left and right boundary.  So in particular, the number $r$ of light particles
is conserved.  When there are no light particles, the two-species
ASEP reduces to the usual ASEP.\footnote{Since the usual ASEP
has been cited as a model for e.g. traffic 
flow, it is natural to ask about ``real life" interpretations
of the two-species ASEP.  
To give one, let us consider
a $1D$ lattice of sites, with a chasm at either side. 
Each site is either empty or occupied by 
a kangaroo or rabbit.  The kangaroos and rabbits can hop left 
and right;  moreover, the 
kangaroos may hop over the rabbits, and over the chasms.
However, the rabbits cannot hop 
over the chasm, so their number is conserved.

To give another practical application, recall that the usual ASEP models the nuclear pore complex,
a multiprotein machine that manages the transport of material into 
and out of the nucleus \cite{NPC1, NPC2}, through a single-file
pore represented by a $1D$ lattice.  One may imagine a generalization
in which the nucleus is semipermeable; some types of 
material can travel in and out, while other types cannot.
}  
Recent work on the two-species ASEP includes 
\cite{Uchiyama, Ayyer1, Ayyer2, Ayyer3,M-V}.
Meanwhile, \emph{homogeneous Koornwinder moments} are defined to 
be the integrals of the homogeneous symmetric polynomials when one uses the  
Koornwinder measure.  (This is a natural multivariate generalization 
of the notion of Askey-Wilson moment
$\mu_n$, which
is the integral of $x^n$ when one uses the Askey-Wilson measure.)
The main result of \cite{CW-Koornwinder} was that the
homogeneous Koornwinder moments  are proportional to the 
fugacity partition function of the two-species ASEP.

In this paper, we introduce \emph{rhombic staircase tableaux}, and show that we have
a trio of relationships  (see Figure \ref{KW_triangle})
\begin{figure}[h]
\centering
\hspace{.2in} \includegraphics[height=.7in]{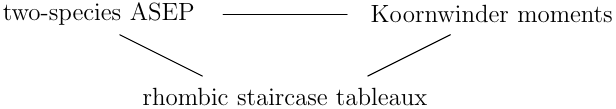}
\caption{}
\label{KW_triangle}
\end{figure}
which is completely parallel to those from Figure \ref{AW_triangle}.
In particular, we give formulas for the steady state probabilities of the two-species ASEP,
and for the homogeneous Koornwinder moments, as generating functions for 
rhombic staircase tableaux, see Theorem \ref{thm:main} and Theorem \ref{thm:moment}. 
It follows from Theorem \ref{thm:moment} that (up to a simple scalar factor) each homogeneous
Koornwinder moment can be written as a polynomial with positive coefficients.

We note that rhombic staircase tableaux 
simultaneously generalize staircase tableaux
(which correspond to the usual ASEP with no light particles)
and also the \emph{rhombic alternative tableaux} of Mandelshtam-Viennot \cite{M-V} (which correspond
to the two-species ASEP when $\gamma=\delta=0$).  Note that 
the lines making up the boxes of a staircase tableaux
are vertical and horizontal.  A beautiful insight of 
\cite{M-V} is that those two kinds of lines 
correspond to particles and holes, and that the combinatorial objects
describing the two-species ASEP should have three kinds of lines,
with vertical and horizontal lines corresponding to heavy particles
and holes as before, and diagonal lines corresponding to light particles,
which are somehow 
a cross between heavy particles and holes.
This led naturally to 
fillings of rhombic tilings.  
See also \cite{k-Mandelshtam} for a generalization of this 
idea to the $k$-species ASEP (when $\gamma=\delta=0$).

There has been some recent work which may be related to ours.
In \cite{BorodinCorwin}, Borodin and Corwin introduced 
\emph{Macdonald processes}, which are
probability measures on sequences of partitions defined in terms of nonnegative specializations
of the (type A) Macdonald symmetric functions, and showed that they are related to various interacting particle systems.
None of the particle systems they discuss is the ASEP (or two-species ASEP) with open boundaries, but perhaps there exists
some more general \emph{Koornwinder processes} which would be connected to the ASEP.

Additionally, the recent work of Cantini
\cite{Cantini} gave a link between the two-species ASEP to Koornwinder polynomials,
with a result very similar to the main result
of \cite{CW-Koornwinder}.
However, Cantini uses the partition
$(1^{N-r},0^r)$, and his techniques are completely different from ours;
e.g. he uses the affine Hecke algebra of type $\hat C_N$ as opposed
to the Matrix Ansatz and the combinatorics of Motzkin paths.

The structure of this paper is as follows.
In Section \ref{sec:2-species}, we define the two-species ASEP. Then
in Section \ref{sec:rhombic}, we introduce 
rhombic staircase tableaux and give our first main result, a combinatorial
formula for steady state probabilities of the two-species ASEP
in terms of rhombic staircase tableaux.
In Section \ref{sec:Koornwinder}, we give our second main result,
a combinatorial formula for homogeneous Koorwinder moments.
In Section \ref{sec:MA} we present a 
\emph{Matrix Ansatz} for the two-species ASEP
(generalizing one given by Uchiyama), which is an algebraic
tool for computing steady state probabilities.
In Section \ref{sec:matrices}, we define a family of matrices,
and explain how they can be thought of as ``transfer matrices''
for rhombic staircase tableaux.
Section \ref{sec:proof1}  
is devoted to the (rather long and technical) 
proof of our first main result.
The proof of our second main result follows reasonably quickly
from the first one, and is given in Section \ref{sec:momentproof}.

\textsc{Acknowledgments:}
We would like to thank Mark Haiman, who pointed out to the
first and third authors in 2007 that
Koornwinder polynomials generalize the Askey-Wilson polynomials, and asked us
if we could make a connection between Koornwinder polynomials and some
generalization of the ASEP.  We would also like to thank Eric Rains,
for several very useful conversations which led to the 
definition of Koornwinder moments.

\section{The two-species ASEP}\label{sec:2-species}

We start by defining the asymmetric exclusion process (ASEP)
with open boundaries.  We will then define the two-species ASEP,
which generalizes the usual ASEP.  

\subsection{The asymmetric exclusion process (ASEP)}

\begin{definition}
Let $\alpha$, $\beta$, $\gamma$, $\delta$,  $q$, and $u$ be constants such that
$0 \leq \alpha \leq 1$, $0 \leq \beta \leq 1$,
$0 \leq \gamma \leq 1$, $0 \leq \delta \leq 1$,
$0 \leq q \leq 1$,
and $0 \leq u \leq 1$.
Let $B_N$ be the set of all $2^N$ words in the
language $\{\circ, \bullet\}^*$.
The \emph{ASEP} is the Markov chain on $B_N$ with
transition probabilities:
\begin{itemize}
\item  If $X = A\bullet \circ B$ and
$Y = A \circ \bullet B$ then
$P_{X,Y} = \frac{u}{N+1}$ (particle hops right) and
$P_{Y,X} = \frac{q}{N+1}$ (particle hops left).
\item  If $X = \circ B$ and $Y = \bullet B$
then $P_{X,Y} = \frac{\alpha}{N+1}$ (particle enters from the left).
\item  If $X = B \bullet$ and $Y = B \circ$
then $P_{X,Y} = \frac{\beta}{N+1}$ (particle exits to the right).
\item  If $X = \bullet B$ and $Y = \circ B$
then $P_{X,Y} = \frac{\gamma}{N+1}$ (particle exits to the left).
\item  If $X = B \circ$ and $Y = B \bullet$
then $P_{X,Y} = \frac{\delta}{N+1}$ (particle enters from  the right).
\item  Otherwise $P_{X,Y} = 0$ for $Y \neq X$
and $P_{X,X} = 1 - \sum_{X \neq Y} P_{X,Y}$.
\end{itemize}
\end{definition}

See Figure \ref{states} for an
illustration of the four states, with transition probabilities,
for the case $N=2$.  The probabilities on the loops
are determined by the fact that the sum of the probabilities
on all outgoing arrows from a given state must be $1$.

\begin{figure}[h]
\centering
\includegraphics[height=1.8in]{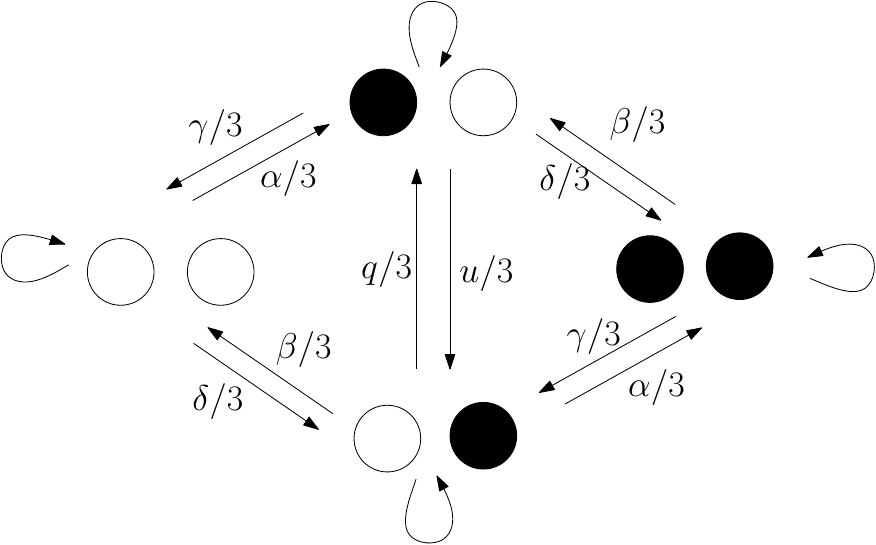}
\caption{The state diagram of the ASEP for $N=2$}
\label{states}
\end{figure}

In the long time limit, the system reaches a steady state where all
the probabilities $P_N(\ttt)$ of finding
the system in configurations $\tau=(\ttt_1, \ttt_2, \dots , \ttt_N)$ are
stationary.  More specifically, the {\it stationary distribution}
is the unique (up to scaling) eigenvector of the transition
matrix of the Markov chain with eigenvalue $1$.

\subsection{The two-species ASEP}

The two-species ASEP is a generalization of the ASEP which involves two kinds of 
particles, \emph{heavy} and \emph{light}.  We will denote a heavy particle by a $2$
and a light particle by a $1$.  We will also denote a hole (or the absence of a particle)
by a $0$.  In the two-species ASEP, heavy particles behave exactly as do
particles in the usual ASEP: they can hop in and out of the boundary, and they can
hop left and right in the lattice (swapping places with a hole or with a heavy particle).
Light particles cannot hop in and out of the boundary, but they may hop
left and right in the lattice (swapping places with a hole or with a light particle).
Therefore the number $r$ of light particles is conserved.

\begin{definition}
Let $\alpha$, $\beta$, $\gamma$, $\delta$,  $q$, and $u$ be constants such that
$0 \leq \alpha \leq 1$, $0 \leq \beta \leq 1$,
$0 \leq \gamma \leq 1$, $0 \leq \delta \leq 1$,
$0 \leq q \leq 1$,
and $0 \leq u \leq 1$.
Let $B_{N,r}$ be the set of all words in
$\{0,1,2\}^N$ which contain precisely $r$ $1$'s;
note that $|B_{N,r}| = {N \choose r} 2^{N-r}$.
The \emph{two-species ASEP} is the Markov chain on $B_{N,r}$ with
transition probabilities:
\begin{itemize}
\item  If $X = A 2 1 B$ and
$Y = A 1 2 B$, or
if $X = A 2 0 B$ and $Y = A 0 2 B$, or
if $X = A 1 0 B$ and $Y = A 0 1 B$,
then
$P_{X,Y} = \frac{u}{N+1}$  and
$P_{Y,X} = \frac{q}{N+1}$.
\item  If $X = 0 B$ and $Y = 2 B$
then $P_{X,Y} = \frac{\alpha}{N+1}$.
\item  If $X = B 2$ and $Y = B 0$
then $P_{X,Y} = \frac{\beta}{N+1}$.
\item  If $X = 2 B$ and $Y = 0 B$
then $P_{X,Y} = \frac{\gamma}{N+1}$.
\item  If $X = B 0$ and $Y = B 2$
then $P_{X,Y} = \frac{\delta}{N+1}$.
\item  Otherwise $P_{X,Y} = 0$ for $Y \neq X$
and $P_{X,X} = 1 - \sum_{X \neq Y} P_{X,Y}$.
\end{itemize}
\end{definition}

See Figure \ref{states2} for an illustration of the state diagram of the 
two-species ASEP on a lattice of $N=2$ sites, with $r=1$ light particle.

\begin{figure}[h]
\centering
\includegraphics[height=1.8in]{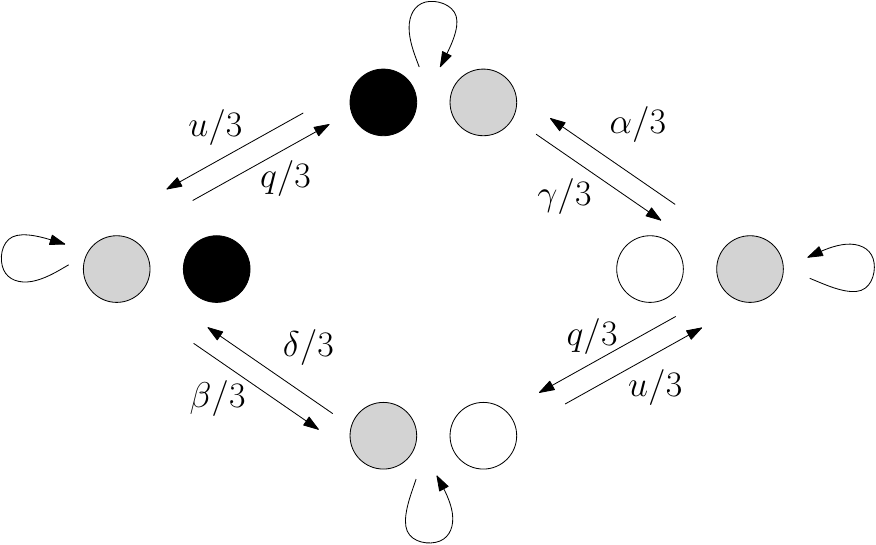}
\caption{The state diagram of the two-species ASEP for $N=2$ and $r=1$.
Light particles are denoted in grey, and heavy particles are denoted
in black.}
\label{states2}
\end{figure}

Note that if $r=0$, i.e. there are no light particles,
then the two-species ASEP is simply the usual ASEP.


\section{Rhombic staircase tableaux 
and the two-species ASEP}\label{sec:rhombic}

In this section we will define \emph{rhombic staircase tableaux}, 
which are the main combinatorial objects used in this paper.
We will then state our main result, which is a formula for
the steady state distribution of the two-species ASEP
in terms of rhombic staircase tableaux.

\begin{definition}
Let $\tau=(\tau_1,\dots,\tau_N)\in \{0,1,2\}^N$.
For $0 \leq i \leq 2$, we 
let $|\tau|_{i}$ denote the number of $i$'s in $\tau$.
We also let $|\tau|_{02}$ denote the total number of $0$'s and $2$'s
in $\tau$.  Set $r = |\tau|_1$.
The \emph{rhombic diagram} $\Gamma(\tau)$ associated to $\tau$
is a convex shape whose southeast border is obtained by 
reading the components of $(\tau_1,\dots,\tau_N)$ from left to 
right, and replacing each $0$ or $2$ by a \emph{corner}
(a south step followed by a west step), and replacing each 
$1$ by a diagonal step southeast.  The northwest border consists
of $N-r$ west steps, followed by $r$ southeast steps, followed
by $N-r$ south steps.
We say that this rhombic diagram has \emph{size} $(N,r)$.
\end{definition}

\begin{example}
Suppose $\tau = (1, 2, 0, 1, 0)$.  Then $N=5$, $r=2$,
and $\Gamma(\tau)$ is as in Figure \ref{shape}.
\begin{figure}[h]
\centering
\includegraphics[height=1.4in]{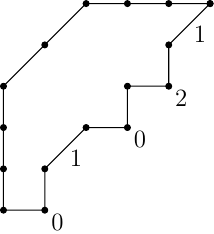}
\caption{The rhombic diagram $\Gamma(1,2,0,1,0)$.}
\label{shape}
\end{figure}
\end{example}

\begin{definition}
A \emph{tile} is a \emph{square}, a \emph{short rhombus}, or a
\emph{tall rhombus}, as shown at the left of Figure \ref{tiles}.
\end{definition}
\begin{figure}[h]
\centering
\includegraphics[height=.6in]{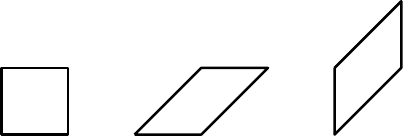} \hspace{2cm}
\includegraphics[height=1.1in]{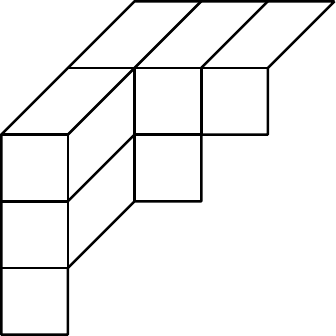}
\caption{At the left: a square, a short rhombus, and a tall rhombus.
At the right: the maximal tiling of $\Gamma(1,2,0,1,0)$.}
\label{tiles}
\end{figure}

\begin{definition}\label{defn:maxtiling}
Let $\tau = (\tau_1,\dots,\tau_N)\in \{0,1,2\}^N$ and 
let $r=|\tau|_1$.
The \emph{maximal tiling} of $\Gamma(\tau)$ is a tiling of
$\Gamma(\tau)$ by squares, short rhombi, and tall rhombi, 
which is obtained as follows. For each $0 \leq j \leq N$,
we construct a column of tiles: 
if $\tau_j=1$, 
this column consists of $|(\tau_1,\dots,\tau_{j-1})|_{02}$
tall rhombi; and   
if $\tau_j = 0$ or $2$,
then this column consists of $1+|(\tau_1,\dots,\tau_{j-1})|_{02}$
squares at the bottom, and $|(\tau_1,\dots,\tau_{j-1})|_1$
short rhombi at the top.  We refer to these columns as
\emph{$1$-columns} and \emph{$02$-columns}, respectively.
We then concatenate these columns from right to left so 
that they comprise a tiling of $\Gamma(\tau)$.  See the rightmost picture in Figure \ref{tiles}.
\end{definition}

\begin{definition}
A \emph{vertical strip} is a maximal connected set of tiles
where two adjacent tiles share a horizontal edge, see 
the first (leftmost) picture in Figure \ref{strips}.  Note 
that 
the $02$-columns from Definition \ref{defn:maxtiling}
are all vertical strips.
A \emph{horizontal strip} is 
a maximal connected set of tiles
where two adjacent tiles share a vertical edge, see the second picture in
Figure \ref{strips}.  The shaded regions in the third and fourth pictures in Figure \ref{strips}
show a vertical and horizontal strip within the maximal tiling of $\Gamma(1,2,0,1,0)$.
\end{definition}
\begin{figure}[h]
\centering
\includegraphics[height=1.1in]{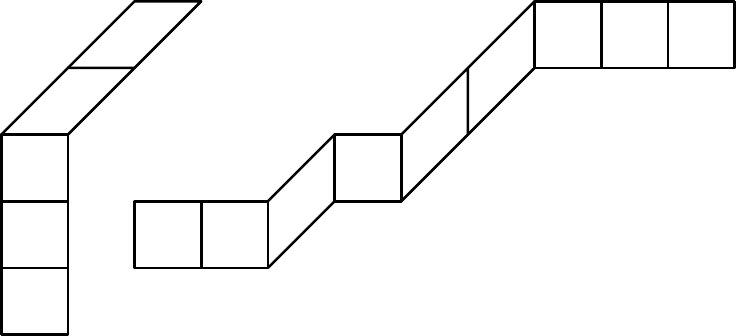} \hspace{1cm}
\includegraphics[height=1.1in]{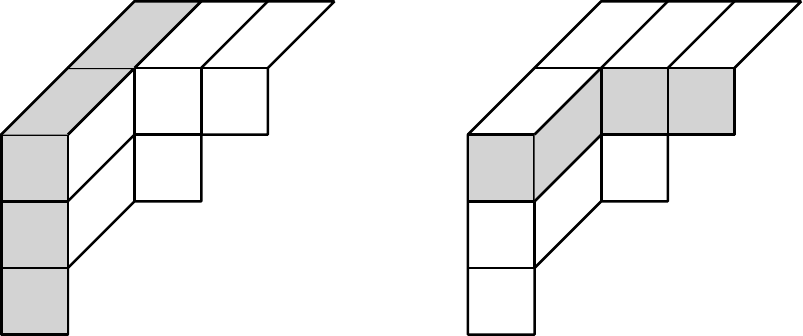} 
\caption{Vertical and horizontal strips.}
\label{strips}
\end{figure}

\begin{remark}
The horizontal strips of the maximal tiling of $\Gamma(\tau)$ where
$\tau = (\tau_1,\dots,\tau_N)$ are in bijection with the corners of 
$\Gamma(\tau)$, or in other words with 
$\{\tau_i \in \tau \ \vert \ \tau_i = 0 \text{ or } 2\}$.
\end{remark}

\begin{definition}
Let $\tau = (\tau_1,\dots,\tau_N)\in \{0,1,2\}^N$.
A \emph{rhombic staircase tableau} $T$ of \emph{shape $\Gamma(\tau)$} 
is a filling of the tiles of the maximal tiling of $\Gamma(\tau)$
such that:
\begin{itemize}
\item Each square is either empty, or 
contains an $\alpha$, $\beta$, $\gamma$, or $\delta$.
\item Each tall rhombus is either empty, or 
contains a $\beta u$ or $\delta q$.
\item Each short rhombus is either empty, or 
contains an $\alpha u$ or $\gamma q$.
\item The lowest square (if there is one) 
in each vertical strip is not empty.
\item Every tile which is above,
and in the same vertical strip as,
 an $\alpha$ or $\gamma$ 
must be empty.
\item Every tile which is left of,
and in the same horizontal strip as,
 a $\beta$ or $\delta$
 must be empty.
\end{itemize}
We say that a rhombic staircase tableau has \emph{size} $(N,r)$
if its shape has size $(N,r)$.
Let us number the components (i.e. the corners and diagonal steps)
of the southeast border from $1$ to $N$.  
If for $1 \leq i \leq N$ we have that 
$\tau_i = 2$ (respectively, $\tau_i=0$) 
if and only if
the corresponding component of the southeast border is a 
corner box containing an $\alpha$ or $\delta$
(resp.\ $\beta$ or $\gamma$),
then we say that the tableau $T$ has 
\emph{type} $\tau$, and we write $\type(T) = \tau$.

\end{definition}
See the left of Figure \ref{filling} for an example of a rhombic staircase tableau
of type $(1,0,0,1,2)$.
\begin{figure}[h]
\centering
\includegraphics[height=1.5in]{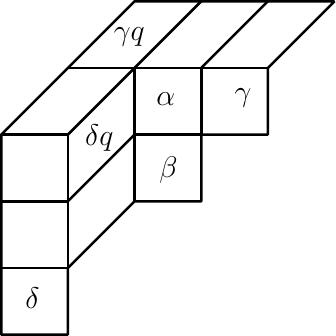} \hspace{1cm}
\includegraphics[height=1.5in]{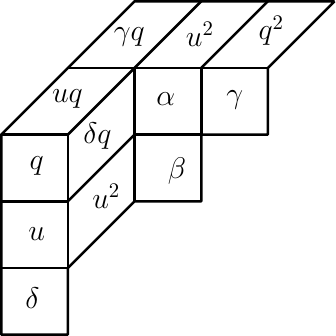} 
\caption{At the left: a rhombic staircase tableau $T$ of 
type $(1,0,0,1,2)$.   
At the right: the same rhombic staircase tableau, with empty tiles filled
according to Definition \ref{defn:weight}.}
\label{filling}
\end{figure}

\begin{definition}\label{defn:weight}
The \emph{weight} $w(T)$ of a rhombic staircase
tableau $T$ is a monomial in $\alpha, \beta, \gamma, \delta, q$,
and $u$, which we obtain as follows.  Every empty tile of $T$ is 
assigned a monomial in $q$ and $u$, based on the label of 
the closest labeled tile to its right in the same 
horizontal strip (if the tile lies in a horizontal strip),
and the label of the closest labeled tile below it in the same
vertical strip (if the tile lies in a vertical strip), such that:
\begin{itemize}[leftmargin=.4cm]
\item each empty square which sees a $\beta$ (resp.\, $\delta$)
to its 
right gets assigned a $u$ (resp. $q$);
\item each empty square which sees an $\alpha$ or $\gamma$
to its right, and an $\alpha$ or $\delta$ below, gets
 a $u$;
\item each empty square which sees an $\alpha$ or $\gamma$
to its right, and a $\beta$ or $\gamma$ below, gets
a $q$;
\item each empty tall rhombus which sees a $\beta$ (resp. $\delta$) to the 
right gets a $u^2$ (resp. $q^2$);
\item each empty tall rhombus which sees an $\alpha$ or $\gamma$ to the 
right gets  a $uq$;
\item each empty short rhombus which sees an $\alpha$ (resp. $\gamma$) below
it gets  a $u^2$ (resp. $q^2$);
\item each empty short rhombus which sees a $\beta$ or $\delta$ below
it gets a $uq$.
\end{itemize}
See Figures \ref{mnemonic1} and \ref{mnemonic2} for a pictorial summary of these rules.
After assigning a monomial in $q$ and $u$ to each empty tile 
in this way, the \emph{weight} of $T$ is then defined as the 
product of all labels in all tiles.

\begin{figure}[h]
\centering
\includegraphics[height=1in]{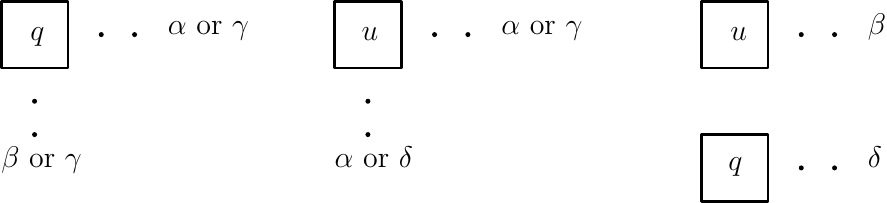} 
\caption{How to assign a $u$ or $q$ to the empty squares of a rhombic staircase tableau.}
\label{mnemonic1}
\end{figure}
\begin{figure}[h]
\centering
\includegraphics[height=.5in]{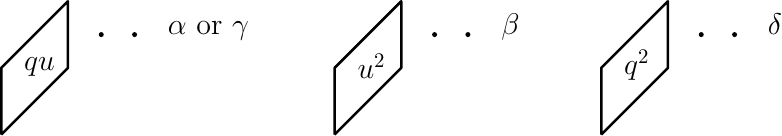} \hspace{.8cm}
\includegraphics[height=.8in]{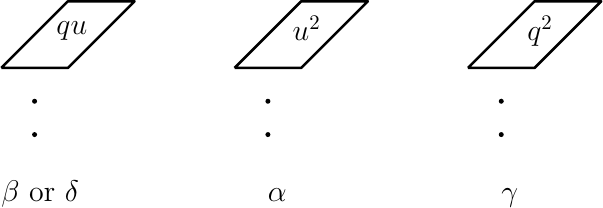} 
\caption{How to assign a monomial in $u$ and $q$ to the empty rhombi of a rhombic staircase tableau.}
\label{mnemonic2}
\end{figure}
\end{definition}
\begin{example}
Let $T$ be the rhombic staircase tableau at the left of Figure \ref{filling}.
Then if we fill the empty boxes according to Definition \ref{defn:weight}, we obtain
the filling shown at the right of Figure \ref{filling}.  We find that 
the weight $w(T)$ of $T$ is 
$\alpha \beta \gamma^2 \delta^2 u^6 q^6$.
\end{example}

\begin{remark}
The weight of a rhombic staircase tableau of size $(N,r)$
always has degree $\frac{(N-r)(N+3r+1)}{2} = {N-r+1 \choose 2} + 2r(N-r)$.  One can see this by using Definition
\ref{defn:maxtiling}, which implies that the maximal tiling of a shape of size $(N,r)$ contains 
${N-r+1 \choose 2}$ squares and $r(N-r)$ rhombi.  
For convenience, we will sometimes set $u=1$, since this results in no loss of information.
\end{remark}

The main result of this paper is the following.  Note that 
it will be a direct consequence of 
Theorem \ref{thm:MA}, 
Theorem \ref{comb-interpret}, 
and Theorem \ref{thm:comb-ansatz}.

\begin{theorem}\label{thm:main}
Consider the two-species ASEP on a lattice of $N$ sites with precisely $r$ 
light particles.
Let $\tau$ be any state; in other words,
$\tau=(\tau_1,\dots,\tau_N) \in \{0,1,2\}^N$, and $|\tau|_1 = r$.  
Set $\mathbf{Z}_{N,r} = \sum_T w(T)$, where the sum is over all rhombic staircase tableaux $T$
of size $(N,r)$.  Then the steady state probability that the two-species ASEP
is at state $\tau$ is precisely $$\frac{\sum_T w(T)}{\mathbf{Z}_{N,r}},$$
where the sum is over all rhombic staircase tableaux $T$ of type $\tau$.
\end{theorem}

\begin{example}
Consider the two-species ASEP on a lattice of $2$ sites with precisely $1$
light particle.  Let $\tau$ be the state $(2,1)$.  There are four tableaux of type
$\tau$; these are shown in Figure \ref{state21}.  Therefore the steady state 
probability that the two-species ASEP is at state $\tau$ is 
equal to $\frac{\alpha \beta u + \alpha \delta q + \alpha u q + \delta q^2}{\mathbf{Z}_{2,1}}.$
\begin{figure}[h]
\centering
\includegraphics[height=.7in]{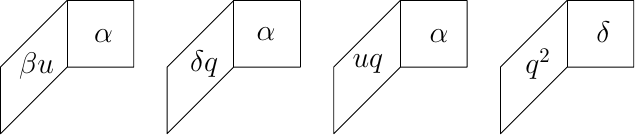} 
\caption{The four tableaux of type $(2,1)$.}
\label{state21}
\end{figure}
\end{example}

\begin{definition}\label{def:fugacity}
We also define 
$\mathbf{Z}_{N,r}(\xi) = \sum_T w(T) \xi^{|T|_2}$, where 
$|T|_2$ is the number of $2$'s in $\type(T)$, and 
the sum is over all rhombic staircase tableaux $T$
of size $(N,r)$.  In other words, $\mathbf{Z}_{N,r}(\xi)$
is the weight generating function for tableaux $T$ of size $(N,r)$, with an extra variable
$\xi$ keeping track of the number of heavy particles in $\type(T)$.

In the case that $r=0$, we write $\mathbf{Z}_N$ and $\mathbf{Z}_N(\xi)$
for $\mathbf{Z}_{N,r}$ and $\mathbf{Z}_{N,r}(\xi)$, respectively.
\end{definition}

\begin{remark}
When $r=0$, rhombic staircase tableaux
recover the staircase tableaux of \cite{CW-Duke1}, and Theorem \ref{thm:main}
specializes to \cite[Theorem 3.5]{CW-Duke1}.
And when $\gamma=\delta=0$, rhombic staircase tableaux 
recover the rhombic alternative tableaux of \cite{M-V},
and Theorem \ref{thm:main} specializes to \cite[Theorem 1.2]{M-V}.
\end{remark}

We have the following closed formula for $\mathbf{Z}_{N,r}$ when $q=1$ (and also $u=1$), 
which follows easily from \cite[Theorem 2.19]{M-V}.
\begin{prop}\label{prop:enumeration}
$$\mathbf{Z}_{N,r}(q=1) ={N \choose r} \prod_{i=r}^{N-1} (\alpha +\beta+\gamma+\delta+i(\alpha+\gamma)(\beta+\delta)).$$ 
\end{prop}
\begin{proof}
When $\gamma=\delta=0$, our rhombic staircase tableaux are 
in bijection with the rhombic alternative
tableaux of Mandelshtam and Viennot \cite{M-V}.  
To obtain a rhombic alternative tableau from a rhombic staircase
tableau which contains no $\gamma$'s or $\delta$'s, simply
delete every diagonal box containing an $\alpha$ 
(respectively, $\beta$) along with 
its vertical strip (respectively, horizontal strip).
So in the case $\gamma=\delta$,
by \cite[Theorem 2.19]{M-V}, we have that 
\begin{equation} \label{MV}
\mathbf{Z}_{N,r}(q=1, \gamma=\delta=0) = {N \choose r} \prod_{i=r}^{N-1} (\alpha + \beta + i\alpha \beta).
\end{equation}
But now note that in the $q=1$ case, the letters $\alpha$ and $\gamma$ play an identical 
role in rhombic staircase tableaux, as  do the letters $\beta$ and $\delta$.  Therefore
we can compute $\mathbf{Z}_{N,r}(q=1)$ by replacing each $\alpha$  in \eqref{MV} by 
$\alpha+\gamma$ and by replacing each $\beta$ by $\beta+\delta$.  The result follows.
\end{proof}

If we set $\alpha=\beta=\gamma=\delta=1$ into Proposition \ref{prop:enumeration}, we obtain 
the following.

\begin{corollary}
The number of rhombic staircase tableaux equals 
$4^{N-r} \frac{N!}{r!} {N \choose r}$.
\end{corollary}

\section{A tableaux formula for Koornwinder moments}\label{sec:Koornwinder}

In this section we will define Koornwinder moments, as in \cite{CW-Koornwinder},
and use rhombic staircase tableaux
to give an explicit formula for the homogeneous Koornwinder moments.
The proof of our formula will 
be given in Section \ref{sec:momentproof}.

\begin{definition}\label{KoornwinderPolynomials}
Let $\zz=(z_1,\dots, z_m)$, $\lambda = (\lambda_1,\dots, \lambda_m)$
be a partition,
and $a, b, c, d, q, t$ be generic complex parameters.
The \emph{Koornwinder polynomials} $P_{\lambda}(\zz; a, b, c, d|q,t)$ are
multivariate orthogonal polynomials which are the type BC-case of Macdonald polynomials.
More specifically,
$P_{\lambda}(\zz; a, b, c, d|q,t)$ is the unique Laurent polynomial which is 
invariant under permutation and inversion of variables, with leading
monomial $\zz^{\lambda}$, and orthogonal with respect to the
\emph{Koornwinder density}
$$ \prod_{1 \leq i < j \leq m}
\frac{(z_i z_j, z_i/z_j, z_j/z_i, 1/{z_i z_j}; q)_{\infty}}
{(t z_i z_j, tz_i/z_j, tz_j/z_i, t/{z_i z_j};q)_{\infty}}
\prod_{1 \leq i \leq m} 
\frac{(z_i^2, 1/{z_i^2}; q)_{\infty}}
{(az_i, a/z_i, bz_i, b/z_i, cz_i, c/z_i, dz_i, d/z_i;q)_{\infty}}$$
on the unit torus
$|z_1|=\dots = |z_m| = 1$, where the parameters satisfy
$|a|, |b|, |c|, |d|, |q|,|t|<1$.

It is well-known that when $q=t$, we have
$$P_{\lambda}(\zz; a, b, c, d| q, q) = \const \cdot
 \frac{\det(p_{m-j+\lambda_j}(z_i; a, b, c, d| q))_{i,j=1}^m}
{\det(p_{m-j}(z_i; a, b, c, d|q))_{i,j=1}^m},$$
where the $p_i$'s are the Askey-Wilson polynomials.
\end{definition}

For orthogonal polynomials in one variable, the $k$th moment $\mu_k$ is defined
to be the integral of $x^k$ with respect
to the associated density function.  For the multivariate Koornwinder polynomials,
there are several ways that we could define moments.  One way would be to
integrate a monomial 
in $x_1,\dots, x_m$ (here we set
$x_i = \frac{z_i+z_i^{-1}}{2}$) 
with respect to the Koornwinder density.
Following a suggestion of Eric Rains \cite{Rains}, we instead
define Koornwinder moments by integrating
Schur polynomials $s_{\lambda}(x_1,\dots,x_m)$ with respect to the Koornwinder density.

\begin{definition}\cite[Definition 2.6]{CW-Koornwinder}
Let $I_k(f(x_1,\dots,x_m);a,b,c,d;q,q)$ denote the result of integrating the
function $f(x_1,\dots,x_m)$ with respect to the Koornwinder density.
Let $\lambda = (\lambda_1,\dots,\lambda_m)$ be a partition.
We define the \emph{Koornwinder moment}
$$M_{\lambda}
= I_k(s_{\lambda}(x_1,\dots,x_m); a, b, c, d; q,q).$$
\end{definition}

We also define another version of Koornwinder moments as follows.

\begin{definition}\label{def:Koornwindermoment}\cite[Definition 4.2]{CW-Koornwinder}
Set $\mathcal{Z}_N(\xi) = 
\prod_{i=0}^{N-1} (\alpha \beta - q^i \gamma \delta)^{-1} \mathbf{Z}_N (\xi)$.\footnote{We need
to pass from $\mathbf{Z}_N$ to $\mathcal{Z}_N$ because in \cite{CW-Koornwinder} we worked
with the version $\mathcal{Z}_N$ of the partition function.}
Given a partition $\lambda = (\lambda_1,\lambda_2,\dots, \lambda_m)$, we define
the \emph{Koornwinder moment} at $q=t$ to be
\begin{equation}\label{def:moment}
K_{\lambda}(\xi) = \frac{\det(\mathcal{Z}_{\lambda_i+m-i+m-j}(\xi))_{i,j=1}^m}{\det(\mathcal{Z}_{2m-i-j}(\xi))_{i,j=1}^m}.
\end{equation}
\end{definition}

The fact that we refer to both quantities $M_{\lambda}$ and $K_{\lambda}(\xi)$ as {Koornwinder moments} is justified by the
following result.

\begin{proposition}\cite[Proposition 4.3]{CW-Koornwinder}
Let $\lambda=(\lambda_1,\dots,\lambda_m)$ be a partition.
We have that $$M_{\lambda}(a,b,c,d|q) = \left(\frac{1-q}{2i}\right)^{|\lambda|} K_{\lambda}(-1; \alpha, \beta, \gamma, \delta;q),$$
where $|\lambda| = \sum_j \lambda_j$, and
 $\alpha, \beta, \gamma, \delta$ are related to $a,b,c,d$ as follows:
\begin{equation}\label{Subs2}
\alpha=\frac{1-q}{1-ac+ai+ci},~~~~~
\beta=\frac{1-q}{1-bd-bi-di},~~~~~
\gamma=\frac{(1-q)ac}{1-ac+ai+ci},~~~~~
\delta=\frac{(1-q)bd}{1-bd-bi-di},
\end{equation}
where $i$ is the square root of $-1$.
\end{proposition}

We refer to moments 
of the form $K_{\lambda}$, where $\lambda = 
(k, 0,  \dots, 0)$ as 
\emph{homogeneous Koornwinder moments}.
Note that in this case, 
$M_{\lambda}$ is the result of integrating the 
homogeneous symmetric function $h_k(x_1,\dots,x_m)$
with respect to the Koornwinder density.  
(Here $m$ is the number of parts of $\lambda$, i.e. 
$m-1$ is the number of $0$'s in the partition.)
This is quite a natural analogue of the notion of moment 
for orthogonal polynomials in one variable.

We note also that the homogeneous Koornwinder moments 
 play a role similar to that of  the homogeneous
symmetric functions: \cite[Theorem 5.1]{CW-Koornwinder} gave a Jacobi-Trudi formula
that expresses a general Koornwinder moment as a determinant in the homogeneous ones.
The following result gives a tableaux formula for the 
homogeneous Koornwinder moments.

\begin{theorem}\label{thm:moment}
The homogeneous Koornwinder moment 
$K_{(N-r,0,0,\dots,0)}(\xi)$ (where there are precisely $r$ $0$'s in the partition)
can be computed in terms of rhombic staircase tableaux via 
$$K_{(N-r,0,0,\dots,0)}(\xi) = \frac{1}{(1-q)^r \prod_{i=0}^{N-r-1} (\alpha \beta - q^{i+2r} \gamma \delta)}
\mathbf{Z}_{N,r}(\xi),$$
where $\mathbf{Z}_{N,r}(\xi)$ is the weight generating function for 
rhombic staircase tableaux given in Definition \ref{def:fugacity}.
\end{theorem}

\section{A Matrix Ansatz for the two-species ASEP}\label{sec:MA}

In this section we will prove Theorem \ref{thm:MA}, which shows that if 
one has matrices and vectors satisfying certain relations, then one can 
use appropriate matrix products to compute steady state probabilities of the 
two-species ASEP.  
Our Matrix Ansatz is a generalization of the one given 
by Uchiyama \cite{Uchiyama}.
We will be working with a distinguished matrix $A$
together with two families of matrices  $\{D^{(i)}\}$ and $\{E^{(i)}\}$ where $i$
ranges over the non-negative integers $\N$.  Note that in this section, we have 
set $u=1$; setting one parameter equal to $1$ does not lose any information.

Before stating the theorem, we need to introduce some terminology.

\begin{definition}\label{def:compatible}
A word $X=X_1\dots X_m$ in $\{A, D^{(i)}, E^{(j}) \ \vert \ i, j \in \N\}$  is 
\emph{$A$-compatible} if $X_k = D^{(i)}$ or $X_k=E^{(i)}$ if and only if there are 
precisely $i$ $A$'s in $X_1 X_2 \dots, X_{k-1}$.

More generally, we say that 
a word $X=X_1\dots X_m$ in $\{A, D^{(i)}, E^{(j)} \ \vert \ i, j \in \N\}$  is 
\emph{$(A,t)$-compatible} if $X_k = D^{(t+i)}$ or $X_k=E^{(t+i)}$ if and only if there are 
precisely $i$ $A$'s in $X_1 X_2 \dots, X_{k-1}$.
\end{definition}

For example, the word $X=D^{(0)} A E^{(1)} A A D^{(3)}$ is $A$-compatible,
and  $X=D^{(3)} A E^{(4)} A A D^{(6)}$ is $(A,3)$-compatible.

\begin{definition}
Let $|X|$ denote the number of letters in $X$, and let $|X|_A$ denote the number of $A$'s in the word $X$.
Set $\| X \| = |X| + |X|_A$.
\end{definition}

\begin{definition}\label{def:Xtau}
Given a word $\tau=(\tau_1,\dots,\tau_N)$ in the letters $\{0,1,2\}$, there is a unique
$A$-compatible word $X(\tau)$ with $|X(\tau)| = N$ which is obtained 
by replacing
each $\tau_i=0$ by an $E^{(j)}$, each $\tau_i=1$ by an $A$, and each $\tau_i=2$ 
by a $D^{(j)}$.
Conversely, given an $A$-compatible word $X$, there is a unique 
sequence $\tau=\tau(X) = (\tau_1,\dots,\tau_N)$ in the letters $\{0,1,2\}$ such that 
$X = X(\tau)$.  
\end{definition}

For example, if $\tau = (0,2,1,1,2,1,0)$, then 
$X(\tau) = E^{(0)} D^{(0)} A A D^{(2)} A E^{(3)}$.
Conversely if 
$X= E^{(0)} D^{(0)} A A D^{(2)} A E^{(3)}$ then 
$\tau(X) = (0,2,1,1,2,1,0)$.

\begin{theorem}\label{thm:MA}
Let $\langle W|$ and $|V\rangle$ be vectors and let 
 $A$, $\{D^{(i)} \ \vert \ i \in \N\}$, and $\{E^{(i)} \ \vert \ i\in \N\}$ be matrices.
Let $\{\lambda_n\}$ be a family of constants indexed by 
the non-negative integers.
Suppose that for any $A$-compatible word $X$ with $|X|_A=t$,  any 
$(A,t)$-compatible word $Y$, and any 
$(A,t+1)$-compatible word $Y'$, the following relations hold:
\begin{enumerate}
\item[(I)] $\langle W|X(D^{(t)} E^{(t)}-q E^{(t)} D^{(t)})Y|V \rangle = 
\lambda_{\| XY \|+2}\langle W| X(D^{(t)}+E^{(t)})Y|V \rangle$
\item[(II)] $\beta \langle W|XD^{(t)}|V \rangle - \delta\langle W|X E^{(t)}|V \rangle= 
\lambda_{\| X \|+1}\langle W| X |V \rangle$
\item[(III)] $\alpha \langle W|E^{(0)} X|V \rangle - \gamma \langle W| D^{(0)} X|V \rangle = 
\lambda_{\| X \|+1}\langle W|X |V \rangle$
\item[(IV)] $\langle W|X(D^{(t)} A-q A D^{(t+1)})Y'|V \rangle = \lambda_{\| XY' \|+3}\langle W| X A Y'|V \rangle$
\item[(V)] $\langle W|X(A E^{(t+1)}-q E^{(t)} A)Y'|V \rangle = \lambda_{\| XY' \|+3}\langle W| X A Y'|V \rangle$
\end{enumerate}
Let $Z_{N,r} = \sum_X \langle W |X | V \rangle$, where the sum is over all 
$A$-compatible words $X$ such that $|X|=N$ and $|X|_A = r$.  
Then in the two-species ASEP on a lattice of $N$ sites with precisely $r$
light particles, the steady state probability of state $\tau=(\tau_1,\dots, \tau_N)$ is equal to
$$\frac{\langle W|X(\tau)|V\rangle}
{Z_{N,r}}.$$
\end{theorem}

For example, the steady state probability of state
$(2,0, 1,1, 2,2)$ is equal to
$\frac{\langle W| D^{(0)} E^{(0)} AAD^{(2)} D^{(2)}|V\rangle}{Z_{6,2}}.$

\begin{proof}
Let $f_N(\tau_1,\dots,\tau_N)$ denote the (unnormalized) steady state probability of 
state $\tau = (\tau_1,\dots,\tau_N)$ of the two-species ASEP,
where $\tau \in \{0,1,2\}^N$  and $\tau$ contains precisely $r$ $1$'s. 
We are in the steady state  if the net rate of entering
each state $\tau$ is $0$, or in other words, the following expression equals $0$:
\begin{align}
&\indicator(\tau_1=0)[-\alpha f_N(0,\tau_2,\dots, \tau_N) + \gamma f_N(2,\tau_2,\dots,\tau_N)] \label{eq:1}\\
+&\indicator(\tau_1=2)[\alpha f_N(0,\tau_2,\dots, \tau_N) - \gamma f_N(2,\tau_2,\dots,\tau_N)] \label{eq:2}\\
+&\sum_{i=1}^{N-1} \indicator(\tau_i=2 \text{ and } \tau_{i+1}=0) [-f_N(\tau_1,\dots,2,0,\dots,\tau_N) + 
                         qf_N(\tau_1,\dots,0,2,\dots,\tau_N)] \label{eq:3}\\
+&\sum_{i=1}^{N-1} \indicator(\tau_i=0 \text{ and } \tau_{i+1}=2) [f_N(\tau_1,\dots,2,0,\dots,\tau_N) - 
                         qf_N(\tau_1,\dots,0,2,\dots,\tau_N)] \label{eq:4}\\
+&\sum_{i=1}^{N-1} \indicator(\tau_i=1 \text{ and } \tau_{i+1}=0) [-f_N(\tau_1,\dots,1,0,\dots,\tau_N) + 
                         qf_N(\tau_1,\dots,0,1,\dots,\tau_N)] \label{eq:5}\\
+&\sum_{i=1}^{N-1} \indicator(\tau_i=0 \text{ and } \tau_{i+1}=1) [f_N(\tau_1,\dots,1,0,\dots,\tau_N) -
                         qf_N(\tau_1,\dots,0,1,\dots,\tau_N)] \label{eq:6}\\
+&\sum_{i=1}^{N-1} \indicator(\tau_i=2 \text{ and } \tau_{i+1}=1) [-f_N(\tau_1,\dots,2,1,\dots,\tau_N) + 
                         qf_N(\tau_1,\dots,1,2,\dots,\tau_N)] \label{eq:7}\\
+&\sum_{i=1}^{N-1} \indicator(\tau_i=1 \text{ and } \tau_{i+1}=2) [f_N(\tau_1,\dots,2,1,\dots,\tau_N) -
                         qf_N(\tau_1,\dots,1,2,\dots,\tau_N)] \label{eq:8}\\
+& \indicator(\tau_N=0) [\beta f_N(\tau_1,\dots, \tau_{N-1},2)-\delta f_N(\tau_1,\dots,\tau_{N-1},0)] \label{eq:9}\\
+& \indicator(\tau_N=2) [- \beta f_N(\tau_1,\dots, \tau_{N-1},2)+\delta f_N(\tau_1,\dots,\tau_{N-1},0)] \label{eq:10}.
\end{align}

Above, we use $\indicator$ to denote the boolean function taking value $1$ or $0$ based on whether its argument
is true or false.  And in \eqref{eq:3} through \eqref{eq:8} above, the arguments
$2,0$ and $0,2$; $1,0$ and $0,1$; and $2,1$ and $1,2$,  are in positions $i$ and $i+1$. 

We need to prove that if we substitute the quantity $\langle W| X(\tau_1,\dots,\tau_n)|V \rangle$
for $f_N(\tau_1,\dots,\tau_N)$, the sum of the expressions in \eqref{eq:1} through \eqref{eq:10} is $0$.
After making this substitution, the assumptions of the theorem imply the following:

\begin{itemize}[leftmargin=.4cm]
\item The sum of \eqref{eq:1} and \eqref{eq:2} is only nonzero if $\tau_1=0$ or $2$, and in that case, 
is equal to $\mp \lambda_{N+r} f_{N-1}(\tau_2,\dots,\tau_N)$, where the sign is chosen based on whether $\tau_1$ is $0$ or $1$.
\item The sum of the $i$th terms of \eqref{eq:3} and \eqref{eq:4} is only nonzero if $\{\tau_i, \tau_{i+1}\} = 
\{0,2\}$, and in that case, equals
$\mp \lambda_{N+r} [ f_{N-1}(\tau_1,\dots,\tau_{i-1},\tau_{i+1},\dots,\tau_N) +
                     f_{N-1}(\tau_1,\dots,\tau_i, \tau_{i+2}, \dots, \tau_N)]$,
based on whether $\tau_i$ equals $2$ or $0$.
\item The sum of the $i$th terms of \eqref{eq:5} and \eqref{eq:6} is only nonzero if $\{\tau_i, \tau_{i+1}\} = 
\{0,1\}$, and in that case, is equal to 
$\mp \lambda_{N+r} [ f_{N-1}(\tau_1,\dots,\tau_{i-1},1,\tau_{i+2},\dots,\tau_N)]$,
based on whether $\tau_i$ equals $1$ or $0$.
\item The sum of the $i$th terms of \eqref{eq:7} and \eqref{eq:8} is only nonzero if $\{\tau_i, \tau_{i+1}\} = 
\{1,2\}$, and in that case, is equal to 
$\mp \lambda_{N+r} [ f_{N-1}(\tau_1,\dots,\tau_{i-1},1,\tau_{i+2},\dots,\tau_N)]$,
based on whether $\tau_i$ equals $2$ or $1$.
\item The sum of \eqref{eq:9} and \eqref{eq:10} is only nonzero if $\tau_N=0$ or $2$, and in that case, 
is equal to $\pm \lambda_{N+r} f_{N-1}(\tau_1,\dots,\tau_{N-1})$ based on whether $\tau_N$ is $0$ or $2$.
\end{itemize}

Then using these conditions, it is easy to verify that the sum of 
\eqref{eq:1} through \eqref{eq:10} is $0$, since all the terms involving $f_{N-1}$ cancel out.
For example,  if $(\tau_1,\dots,\tau_N) = (0,2,1,0,2)$ (so that $N=5$ and $r=1$), we get that the sum
of \eqref{eq:1} through \eqref{eq:10} equals:
\begin{eqnarray*}
&- \lambda_6 f_4(2,1,0,2) \hspace{1cm}  &\text{ (from \eqref{eq:1} and \eqref{eq:2})} \\
&+  \lambda_6 [ f_4(2,1,0,2)+f_4(0,1,0,2)] \hspace{1cm} &\text{ (from the $i=1$ terms of \eqref{eq:3} and \eqref{eq:4}) }\\
&-  \lambda_6  f_4(0,1,0,2) \hspace{1cm} &\text{ (from the $i=2$ terms of \eqref{eq:7} and \eqref{eq:8}) }\\
&-  \lambda_6  f_4(0,2,1,2) \hspace{1cm} &\text{ (from the $i=3$ terms of \eqref{eq:5} and \eqref{eq:6}) }\\
&+  \lambda_6  [f_4(0,2,1,2)+f_4(0,2,1,0)] \hspace{1cm} &\text{ (from the $i=4$ terms of \eqref{eq:3} and \eqref{eq:4}) }\\
&-  \lambda_6  f_4(0,2,1,0) \hspace{1cm} &\text{ (from \eqref{eq:9} and \eqref{eq:10}) }
\end{eqnarray*}
Clearly the sum of these terms is $0$.
\end{proof}

\section{The matrices describing rhombic staircase tableaux}\label{sec:matrices}

In this section we will define a row vector $\langle W|$, a column vector $|V \rangle$, a distinguished matrix $A$,
and two families of matrices $\{D^{(t)} \ \vert \ t\in \N\}$ 
and $\{E^{(t)} \ \vert \ t\in \N\}$.  We will then explain how appropriate matrix products compute generating functions for 
rhombic staircase tableaux.

\subsection{The definition of the matrices}

\begin{definition}\label{def:matrices}
In particular, our matrices and 
vectors are not finite.  We define ``row'' and ``column'' vectors 
$\langle W| = (W_{ik})$ and $|V\rangle = (V_{j\ell})$,
where the indices $i, j, k, \ell$ 
range over the non-negative integers $\N$. For each $t\in \N$,
we also define  matrices 
$D^{(t)} = (D^{(t)}_{i,j,k,\ell})$ and $E^{(t)} = (E^{(t)}_{i,j,k,\ell})$,
as well as the matrix $A = (A_{i,j,k,\ell})$ 
(where again the indices $i,j,k,\ell$ range over $\N$)
by the following:
\begin{eqnarray*}
W_{ik}&=&\left\{
\begin{array}{ll}
1 &\textup{if $i=k=0$,}\\
0  &\textup{otherwise,}
\end{array}  \right.\\
V_{j\ell}&=&1 \text{ always.}
\end{eqnarray*}
\[
D^{(t)}_{i,j,k,\ell} = \begin{cases}
0 & \mbox{if } j<i \mbox{ or } \ell>k+1 \\
\alpha q^i &\mbox{if } k=0, \ell=1, \mbox{ and } i=j\\
\delta q^i(q^t+(\alpha+\gamma q^t)[t]_q) &\mbox{if } k=0,\ell=0, \mbox{ and } j=i+1\\
\delta(D^{(t)}+E^{(t)})_{i,j-1,k-1,\ell}+D^{(t)}_{i,j,k-1,\ell-1} & \mbox{otherwise.}
\end{cases}
\]
\[
E^{(t)}_{i,j,k,\ell} = \begin{cases}
0 & \mbox{if } j<i \mbox{ or } \ell>k+1 \\
\gamma q^{2t+i} &\mbox{if } k=0, \ell=1, \mbox{ and } i=j\\
\beta q^i(q^t+(\alpha+\gamma q^t)[t]_q) &\mbox{if } k=0,\ell=0, \mbox{ and } i=j\\
\beta(D^{(t)}+E^{(t)})_{i,j,k-1,\ell}+q E^{(t)}_{i,j,k-1,\ell-1} & \mbox{otherwise.}
\end{cases}
\]
\[
A_{i,j,k,\ell} = \begin{cases}
0 & \mbox{if } \ell>k \mbox{ or }j-i>k-\ell\\
q^{2i} & \mbox{if } i=j \mbox{ and }k=\ell=0\\
\beta A_{i,j,k-1,\ell}+\delta q A_{i,j-1,k-1,\ell} + q A_{i,j,k-1,\ell-1} & \mbox{otherwise.}
\end{cases}
\]
In the formulas above we use $[t]_q$ to denote the \emph{$q$-analogue} 
$1+q+\dots +q^{t-1} = \frac{q^t-1}{q-1}$.  And 
by convention, if any subscript $i,j,k$ or $\ell$ is negative, then 
the matrix entry $M_{ijk\ell} = 0$.
\end{definition}

When we write $M_{i,j,k,\ell}$, we are thinking of the two coordinates $i$ and $k$ as specifying a ``row'' of a matrix,
and the two coordinates $j$ and $\ell$ as specifying a ``column'' of a matrix.  Therefore
matrix multiplication is defined by 
$$(MM')_{i,j,k,\ell} = \sum_{a,b} M_{i,a,k,b} M'_{a,j,b,\ell},$$
where $a$ and $b$ vary over $\N$.
Note that
when $M$ and $M'$ are products of $D^{(t)}$'s, $E^{(t)}$'s, and $A$'s, the sum on the right-hand-side
is  finite.
Specifically, if $M$ is a product of $s$ factors as above, 
then $M_{i,a,k,b}$ is $0$ unless $a+b \leq i + k+s$.  This can
be shown by induction from the definition of the matrices,
or by using the combinatorial interpretation
given in Lemma \ref{comb-lemma}.

Although we don't need it in this paper, we note that one can use the recursive definition of $A$ to prove the following explicit formula for 
$A_{i,j,k,\ell}$.
\begin{rem}
\[
A_{i,j,k,\ell} = \delta^{j-i}\beta^{k-\ell-(j-i)}q^{\ell+i+j} {k \choose \ell}{k-\ell \choose j-i}
\]
\end{rem}

\subsection{The combinatorial interpretation of our matrices in terms
of tableaux}\label{sec:comb-interpret}

We say that a horizontal strip of a rhombic staircase tableau $T$
is a horizontal {\it $\beta$-strip} if the leftmost tile in that strip 
containing a Greek letter contains a $\beta$ (or $\beta u$).
Note that every tile to the left of that $\beta$
must contain a $u$ or $u^2$.  Similarly we will talk about horizontal 
\emph{$\delta$-strips};
in this case, every tile to the left of that
$\delta$ must contain a $q$ or $q^2$.  We will also talk about horizontal 
\emph{$\alpha/\gamma$}-strips,
which is shorthand
for horizontal strips which are $\alpha$-strips {\it or} $\gamma$-strips.


Recall the notions of $A$-compatible and $(A,t)$-compatible from Definition \ref{def:compatible}.

\begin{theorem}\label{comb-interpret}
Suppose that $X$ is an $(A,t)$-compatible word 
in $\{A, D^{(i)}, E^{(j)} \ \vert \ i,j\in \N\}$.
Set $u=1$.
Then we have that:
\begin{enumerate}

\item $X_{ijk\ell}$ is the (weight) generating function for
all ways of adding $|X|$ new columns of type $\tau(X)$  to a rhombic staircase
tableau with precisely $t$ $1$-columns, 
$i$ horizontal $\delta$-strips, and $k$ horizontal 
$\alpha/\gamma$-strips,
so as to obtain a new tableau with $j$ horizontal $\delta$-strips 
$\ell$ horizontal $\alpha/\gamma$-strips;

\item if $t=0$, then $(WX)_{j\ell}$ is the generating function for
rhombic staircase tableaux of type $\tau(X)$ which have $j$ horizontal 
$\delta$-strips
and $\ell$ horizontal $\alpha/\gamma$-strips;

\item if $t=0$, then $\langle W|X|V\rangle$ is the  generating function for all rhombic staircase
tableaux of type $\tau(X)$.

\end{enumerate}
\end{theorem}

Recall 
the definition of $\mathbf{Z}_{N,r}(\xi)$
from Definition \ref{def:fugacity}. 

\begin{corollary}\label{cor:ZNr}
We have that $\mathbf{Z}_{N,r}(\xi) = 
\sum_X \xi^{|X|_D} \langle W| X | V \rangle$, where the sum is 
over all $A$-compatible words $X$ such that $|X| = N$ and 
$|X|_A = r$.
\end{corollary}

The main step in proving Theorem \ref{comb-interpret}
is the
following lemma, which
says that the matrices $A$, $D^{(i)}$ and $E^{(j)}$
are ``transfer matrices'' for building
rhombic staircase tableaux.  

\begin{lemma}\label{comb-lemma}
$A_{i,j,k,\ell}$ is 
the generating function for the weights of all possible 
new $1$-columns that we could add to the left of a 
rhombic staircase tableau with $i$ horizontal $\delta$-strips 
and $k$ horizontal $\alpha/\gamma$-strips, so as to obtain a new
tableau which has $j$ horizontal $\delta$-strips and $\ell$ horizontal 
$\alpha/\gamma$-strips.

$D^{(t)}_{i,j,k,\ell}$ is the generating function for the weights
of all possible new $02$-columns with an $\alpha$ or $\delta$ in the bottom box
that we could add to the left of a rhombic staircase tableau with precisely $t$ $1$-columns,
$i$ horizontal $\delta$-strips, and $k$ horizontal $\alpha/\gamma$-strips,
so as to obtain a tableau which has $j$ horizontal $\delta$-strips
and $\ell$ horizontal $\alpha/\gamma$-strips.
Similarly for $E^{(t)}_{i,j,k,\ell}$, where the new $02$-column has a $\beta$
or $\gamma$ in the bottom box.
\end{lemma}

\begin{proof}
Let $\tilde{A}_{ijk\ell}$ denote 
the generating function for the weights of all possible 
new $1$-columns that we could add to the left of a 
rhombic staircase tableau with $i$ horizontal $\delta$-strips 
and $k$ horizontal $\alpha/\gamma$-strips, so as to obtain a new
tableau which has $j$ horizontal $\delta$-strips and $\ell$ horizontal 
$\alpha/\gamma$-strips.
We will show that $\tilde{A}_{ijk\ell} = A_{ijk\ell}$ by showing that 
$\tilde{A}$ satisfies the same recurrences.

Note that $\tilde{A}_{ijk\ell}=0$ if $\ell>k$ because a $1$-column consists of 
tall rhombi, and a tall rhombus may never contain an $\alpha$ or $\gamma$.
Now consider a potential $1$-column that contributes to $\tilde{A}_{ijk\ell}$, and 
let $h$ be the number of its tiles containing a $\delta$.  Then 
$h=j-i$, but also $k-\ell \geq h$.  It follows that $k-l\geq j-i$, and hence 
$\tilde{A}_{ijk\ell} = 0$ if $j-i>k-\ell$.

Now suppose that $i=j$ and $k=\ell=0$.
Then the horizontal strips of the 
$1$-columns contributing to $\tilde{A}_{ijk\ell}$ 
are all $\beta$-strips or $\delta$-strips.  
So in particular, each tile must either 
contain a $u^2=1$ or 
a $q^2$, based on whether its horizontal strip is a $\beta$-strip 
or a $\delta$-strip.
It follows that there is only one column, and it has weight $q^{2i}$.

Finally suppose that $k>0$.  (Note that if $k=0$ then this implies 
that $\ell=0$ and $j\geq i$, in which case we are in one of the cases considered previously.)
Let us consider the different ways that we can fill in 
the lowest tile $B$ 
of a $1$-column contributing to $\tilde{A}_{ijk\ell}$, such that 
$B$ belongs to an $\alpha/\gamma$-strip.
(Since $k>0$, such a tile exists.)
Either $B$ contains a $\beta u$ or a $\delta q$ or 
is empty, in which case Figure \ref{mnemonic2} implies that it 
gets assigned a $qu$.
If $B$ contains a $\beta u$, then the process of filling in the other tiles has a 
generating function of $\tilde{A}_{i,j,k-1,\ell}$, so we get a contribution of 
$\beta \tilde{A}_{i,j,k-1,\ell}$ (recall that we are setting $u=1$).
If $B$ contains a $\delta q$, then filling in the other tiles  
has generating function $\tilde{A}_{i,j-1,k-1,\ell}$, so we get a contribution of 
$\delta q \tilde{A}_{i,j-1,k-1,\ell}$.  And if $B$ contains a $qu$,
then filling in the other tiles has generating function $\tilde{A}_{i,j,k-1,\ell-1}$,
so we get a contribution of $q \tilde{A}_{i,j,k-1,\ell-1}$.  

This shows that
if $k>0$, then $\tilde{A}_{i,j,k,\ell} = \beta \tilde{A}_{i,j,k-1,\ell} + 
\delta q \tilde{A}_{i,j-1,k-1,\ell} + q\tilde{A}_{i,j,k-1,\ell-1}$.

Let $\tilde{D}^{(t)}_{ijk\ell}$ denote the generating function for $02$-columns
described in the second sentence of 
Lemma 
\ref{comb-lemma}.  As before, we want to show that $\tilde{D}^{(t)}$ satisfies the defining 
recurrence for $D^{(t)}$.

Note that $\tilde{D}^{(t)}_{ijk\ell}=0$ if $j<i$ because
adding a new column to a rhombic staircase tableau never
decreases the number of $\delta$-strips.
Also $\tilde{D}^{(t)}_{ijk\ell}=0$ if $\ell > k+1$ because when
we add a new column we can never increase the number
of $\alpha/\gamma$-strips by more than $1$.

Now suppose that $k=0$.  Since 
we are starting from
a tableau with 
$0$ horizontal $\alpha/\gamma$-strips,
it follows that any new $02$-column that we can add to the tableau will 
have either an $\alpha$ or $\delta$ in the bottom square,
with all squares above that one empty, and with $t$ short rhombi on top 
of the squares.  If we put an $\alpha$ in the bottom square, then 
the short rhombi in that column must also be empty, so the weight of 
the resulting column will be $\alpha q^i$.  The resulting tableau will have
$\ell=1$ horizontal $\alpha/\gamma$-strips and $j=i$ horizontal
$\delta$-strips.
On the other hand, if we put a $\delta$ in the bottom square of the new column,
the weight of the \emph{squares} in the column will be $\delta q^i$.  
As for the short rhombi, we can either leave them all empty (so that they get
weight $(qu)^t = q^t$), or 
we can put an $\alpha u$ or $\gamma qu$ into one of them.
There are $t$ short rhombi, and the weight of the $t$ different ways 
of putting an $\alpha u$ into one of them is $\alpha (1+q+\dots +q^{t-1})$.
The weight of the $t$ different ways of putting a $\gamma q$ into one of them
is $\gamma q^t(1+q+\dots+q^{t-1})$.  
Therefore the weight of the various columns we can add with a $\delta$ in the bottom square
is $\delta q^i (q^t+(\alpha + \gamma q^t)[t]_q)$.
And any resulting tableau will have $\ell=0$ horizontal  
$\alpha/\gamma$-strips, and $j=i+1$ horizontal $\delta$-strips.

From this discussion it follows that
$\tilde{D}^{(t)}_{ijk\ell}=\alpha q^i$ when $k=0$, $\ell=1$, and $j=i$,
and $\tilde{D}^{(t)}_{ijk\ell}=\delta q^i(q^t+(\alpha+\gamma q^t)[t]_q)$ when $k=\ell=0$ and $j=i+1$.

In all other situations, we can assume that $k \geq 1$.
Suppose that we are adding a new column $C$
with an $\alpha$ or $\delta$ at the bottom to the left of a rhombic staircase
tableau with $i$ horizontal $\delta$-strips and $k$
horizontal $\alpha/\gamma$-strips, so as to create a new tableau $T$.
Consider the lowest square $B$ of $C$ 
which belongs to 
an $\alpha/\gamma$-strip in $T$ (such a square exists since $k \geq 1$).
If we fill $B$ with an $\alpha, \beta, \gamma$ or $\delta$,
then the bottom square of $C$ must contain a $\delta$ (not an $\alpha$).
In this case, if we ignore that bottom $\delta$,
then our choices for $C$ are exactly
the same as our choices would be for adding a
new column to the left of a tableau
with $i$ horizontal $\delta$-strips and $k-1$
horizontal $\alpha/\gamma$-strips.
Therefore, filling $B$ with an $\alpha, \beta, \gamma$ or $\delta$
gives us a contribution of
$\delta (\tilde{D}^{(t)}+\tilde{E}^{(t)})_{i,j-1,k-1,\ell}$ to our generating function.

On the other hand, if we leave $B$ empty,
then this square will get a weight $u=1$.
Filling the rest of the column $C$ is like
adding a new column to a  tableau
with $i$ horizontal $\delta$-strips and $k-1$ horizontal 
$\alpha/\gamma$-strips.  Therefore leaving $B$ empty
gives us a contribution of
$\tilde{D}^{(t)}_{i,j,k-1,\ell-1}$ to our generating function.

It follows that when $k \geq 1$,
$\tilde{D}^{(t)}_{i,j,k,\ell} = \delta (\tilde{D}^{(t)}+\tilde{E}^{(t)})_{i,j-1,k-1,\ell} 
+ \tilde{D}^{(t)}_{i,j,k-1,\ell-1}.$

The proof of the statement of the lemma concerning $E^{(t)}_{i,j,k,\ell}$
is 
analogous to the proof we gave for $D^{(t)}_{i,j,k,\ell}$.
\end{proof}

\begin{proof}[Proof of Theorem \ref{comb-interpret}]
The first item follows from Lemma \ref{comb-lemma} and the definition
of matrix multiplication.
Assuming that $t=0$, multiplying at the left by $\langle W|$ has the effect that we start
with the empty tableau and then add columns according to $X$:
so
$(W X)_{j\ell}$ is the generating function
for rhombic staircase tableaux of type $\tau(X)$ which have
$j$ horizontal $\delta$-strips and $\ell$ horizontal 
$\alpha/\gamma$-strips.  
Finally, multiplying $\langle W |X$ on the right by $|V\rangle$
has the effect of summing over all $\delta$ and $\ell$,
so $\langle W|X|V\rangle$ is the generating function for all rhombic staircase
tableaux of type $\tau(X)$.
\end{proof}

\section{The proof that our matrices satisfy the Matrix Ansatz}\label{sec:proof1}

Recall that in Section \ref{sec:matrices} we defined some matrices 
and vectors, and showed in Theorem \ref{comb-interpret} (3) that 
appropriate matrix products compute generating functions
for rhombic staircase tableaux.  In this section, we will 
show in Theorem \ref{thm:comb-ansatz} 
that these matrix products also 
satisfy the Matrix Ansatz relations of Theorem \ref{thm:MA}.  This 
implies that these matrix products also 
compute steady 
state probabilities for the two-species ASEP.  Our main theorem,
Theorem \ref{thm:main}, will therefore follow 
once we have proved Theorem \ref{thm:comb-ansatz}.

\begin{thm}\label{thm:comb-ansatz}
The matrices and vectors from
Definition \ref{def:matrices} satisfy 
relations (I)-(V) of Theorem \ref{thm:MA}, 
with
$\lambda_n = \alpha\beta - \gamma\delta q^{n-1}$. 
\end{thm}

The proof of Theorem \ref{thm:comb-ansatz} is long 
and technical, but we will try to make the
structure of the proof transparent from the titles
of the following subsections.

\subsection{Reducing  Theorem \ref{thm:comb-ansatz}
to (2) and (4) from Proposition 
\ref{prop:refined}}

In this section we show that Theorem \ref{thm:comb-ansatz} is 
a consequence of 
equations (2) and (4) from Proposition
\ref{prop:refined}.

\begin{lem}
Equation (III) of Theorem \ref{thm:MA} holds.
\end{lem}
\begin{proof}
We use a simple combinatorial argument to show 
(III), namely that for any $A$-compatible word $X$, we have 
that 
$$\alpha \langle W|E^{(0)} X|V \rangle - \gamma \langle W| D^{(0)} X|V 
\rangle = 
\lambda_{\| X \|+1}\langle W|X |V \rangle.$$
By Theorem \ref{comb-interpret}, we can interpret 
each of the terms in this equation in terms of rhombic 
staircase tableaux, for example 
$ \langle W|E^{(0)} X|V \rangle$ 
(respectively, 
$\langle W| D^{(0)} X|V \rangle$)
is the generating function 
for  tableaux of type $\tau(E^{(0)}X)$ (resp. $\tau(D^{(0)}X)$).
Such a tableau has a $\beta$ or $\gamma$ (resp. $\alpha$ or 
$\delta$) in 
its upper-right corner box.

See the first four terms of 
Figure \ref{fig:comb} for a pictorial representation 
of the quantity
$\alpha \langle W|E^{(0)} X|V \rangle - \gamma \langle W| D^{(0)} X|V 
\rangle$.   Note that the $u$'s in the first term, and the 
$q$'s in the fourth term, are forced by the presence of the 
$\beta$ and $\delta$, respectively.  All tiles left blank 
can be filled arbitrarily, so as to create a rhombic staircase
tableaux of the appropriate type.
So e.g. the first term
represents the sum of the weights of all tableaux of type 
$\tau(E^{(0)}X)$ which 
can be obtained by filling in its blank tiles. 

\begin{figure}[h]
\centering
\includegraphics[height=1in]{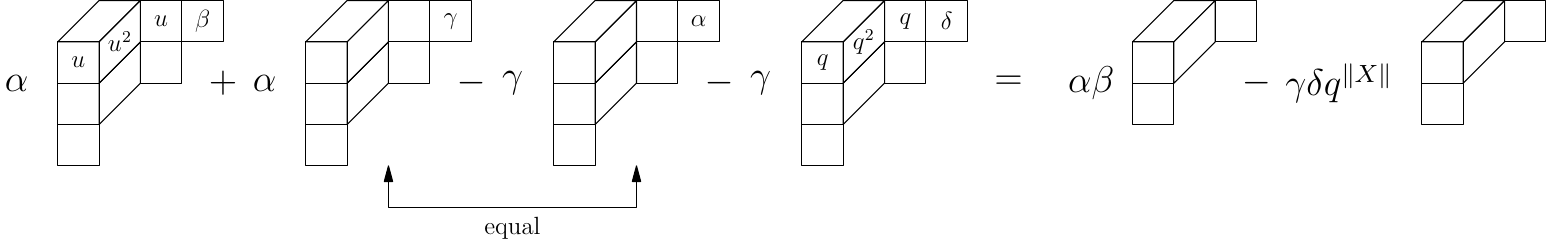}
\caption{The combinatorial proof of (III)}
\label{fig:comb}
\end{figure}

It is easy to see that the second and third terms in
Figure \ref{fig:comb} are equal to each other and hence 
cancel out.  Moreover, 
one may delete the topmost horizontal strip from each of the
first and fourth terms, so as to obtain the fifth and sixth terms
in  Figure \ref{fig:comb}.  (Remember that we set $u=1$).
To complete the proof, we simply observe that the 
sum of the fifth and sixth terms in Figure \ref{fig:comb}
is equal to 
$\lambda_{\| X \|+1}\langle W|X |V \rangle.$
\end{proof}

To prove the other equations from Theorem \ref{thm:MA}, we
reduce them to the more refined statements of 
Proposition \ref{prop:refined}.
First we prove the following useful lemma.

\begin{lem}\label{ij_reduction}
For any $(A,t)$-compatible word $Y$  
in $\{A, D^{(i)}, E^{(j)} \ \vert \ i,j\in \N\}$,
we have that $Y_{i,j,k,\ell} = q^{\| Y \|}Y_{i-1,j-1,k,\ell}$.
\end{lem}

\begin{proof}
To prove the lemma we use the combinatorial interpretation 
given by Lemma \ref{comb-interpret}.
Note that both $Y_{i,j,k,\ell}$ and 
$Y_{i-1,j-1,k,\ell}$ enumerate the ways of adding
$|Y|$ new columns to a fixed tableau $T$ so as to increase by 
$j-i$ the number of horizontal $\delta$-strips,
and to increase by $\ell-k$ the number of horizontal 
$\alpha/\gamma$-strips.  The only difference is the 
initial number of horizontal $\delta$-strips.
Since $Y_{i,j,k,\ell}$ has one extra initial horizontal 
$\delta$-strips, this will contribute $|Y|$ extra 
empty tiles which each get a weight of $q$ or $q^2$.
More specifically, there will be $|Y|-|Y|_A$ square tiles
which each get a $q$, and $|Y|_A$ tall rhombi which each get a $q^2$.
Therefore, the difference in weight will be 
$q^{\| Y \|}$.
\end{proof}

\begin{prop}\label{prop:refined}
Let $r$ be a non-negative integer and let 
$X$ be any $A$-compatible word 
in $\{A, D^{(i)}, E^{(j)} \ \vert \ i,j\in \N\}$
containing precisely $r$ A's.
To prove equations (I), (II), (IV), and (V) from Theorem \ref{thm:MA},
it suffices to prove the following:
\begin{enumerate}
\item[(1)] 
$(WX D^{(r)}E^{(r)})_{j,\ell} = q(W X E^{(r)}D^{(r)})_{j,\ell}+\alpha\beta(WX(D^{(r)}+E^{(r)}))_{j,\ell}\\
\phantom{helloaaaaaaaaaaaaaaaaaaaaaaaaaaa} -\gamma\delta q^{\| X \|+1}(WX(D^{(r)}+E^{(r)}))_{j-1,\ell}$
\item[(2)] $\beta(WX D^{(r)})_{j,\ell} = \delta(W X E^{(r)})_{j-1,\ell}+\alpha\beta(WX)_{j,\ell-1}-\gamma\delta q^{\| X \|}(WX)_{j-1,\ell-1}$
\item[(4)] $(WX D^{(r)} A)_{j,\ell} = q(W X A D^{(r+1)})_{j,\ell}+\alpha\beta(WX A)_{j,\ell}-\gamma\delta q^{\| X \|+2}(WX A)_{j-1,\ell}$
\item[(5)] $(WX A E^{(r+1)})_{j,\ell} = q(W X E^{(r)} A)_{j,\ell}+\alpha\beta(WX A)_{j,\ell}-\gamma\delta q^{\| X \|+2}(WX A)_{j-1,\ell}$
\end{enumerate}
\end{prop}

\begin{proof}
We will show that (1) implies (I), (2) implies (II), (4) implies
(IV), and (5) implies (V).

First note that if we take (2) and sum over all non-negative
$j$ and $\ell$, we immediately obtain (II), because 
multiplication on the right by $|V \rangle$ has the effect
of summing over all indices.

To show that (1) implies (I), we first claim that 
if (1) is true, then for any $(A,r)$-compatible word $Y$,
$(WXD^{(r)}E^{(r)}Y)_{j,\ell}$ is equal to
$$q(WXE^{(r)}D^{(r)}Y)_{j,\ell} +
\alpha \beta (WX(D^{(r)}+E^{(r)})Y)_{j,\ell}-\gamma \delta q^{\|X\|+\|Y\|+1}
 (WX(D^{(r)}+E^{(r)})Y)_{j-1,\ell}.$$
To prove the claim, note that: 
 $(WXD^{(r)}E^{(r)} Y)_{j\ell}$ 
is equal to: 
{\small
\begin{eqnarray*}
&& \sum_{i,k} (WXD^{(r)}E^{(r)})_{i,k} Y_{i,j,k,\ell}\\
&=&\sum_{i,k} q(WXE^{(r)}D^{(r)})_{i,k} Y_{i,j,k,\ell} + 
      \alpha \beta \sum_{i,k} (WX(D^{(r)}+E^{(r)}))_{i,k} Y_{i,j,k,\ell}\\
    &&
- \gamma \delta q^{\|X\|+1}\sum_{i,k}
      (WX(D^{(r)}+E^{(r)}))_{i-1,k} Y_{i,j,k,l} \\
&=&q (WXE^{(r)}D^{(r)}Y)_{j,\ell} + \alpha \beta(WX(D^{(r)}+E^{(r)})Y)_{j\ell} -
      \gamma \delta q^{\|X\|+\|Y\|+1} (WX(D^{(r)}+E^{(r)})Y)_{j-1,\ell}.
\end{eqnarray*}
}
To deduce the final equality above, we applied Lemma \ref{ij_reduction}
to the last term.

Now note that if we take the equation of the claim, and sum over all
$j$ and $\ell$, then we get precisely (I).

The proof that (4) implies (IV) and the proof that (5) implies
(V) are completely analogous.  For example, to prove the former,
we first show that (4) implies that for any 
$(A,r+1)$-compatible word $Y'$, we have that 
$$(WXD^{(r)}AY')_{j,\ell} = 
q (WXAD^{(r+1)}Y')_{j,\ell} + 
\alpha \beta (WXAY')_{j,\ell} 
- \gamma \delta q^{\|X\|+\|Y\|+2} (WXAY')_{j-1,\ell}.$$
We then sum this equation over all $j$ and $\ell$ to obtain 
(IV).
\end{proof}

To complete the proof of Theorem \ref{thm:comb-ansatz}, 
we need to prove that  equations (1), (2), (4), and (5)
from Proposition \ref{prop:refined} hold.
In Lemmas \ref{lem:24-5} and \ref{lem:2-1} below, we will 
show that (2) and (4) imply (5), and also that (2) implies (1).
The rest of the proof will be devoted to proving 
equations (2) and (4).

\begin{lem}\label{lem:24-5}
If equations (2) and (4) 
from Proposition \ref{prop:refined} hold, then so 
does equation (5).
\end{lem}

\begin{proof}
We use the definition of matrix multiplication to write
\[
q(WX E^{(r)} A)_{j,\ell} = q \sum_{i,k} (WX E^{(r)})_{i,k} A_{i,j,k,\ell}.
\]
Using (2) to expand 
$(WX E^{(r)})_{i,k}$, we get the following expression for the 
right-hand side:
\begin{eqnarray*}
&& q \sum_{i,k} {\left(\frac{\beta}{\delta}(W X D^{(r)})_{i+1,k} - \frac{\alpha\beta}{\delta}(W X)_{i+1,k-1} + \gamma q^{\| X \|}(W X)_{i,k-1} \right)} A_{i,j,k,\ell}
\\
&=& q \sum_{i,k} \left( \frac{\beta}{\delta q^2}(W X D^{(r)})_{i+1,k}A_{i+1,j+1,k,\ell} - \frac{\alpha\beta}{\delta q^2}(W X)_{i+1,k-1}A_{i+1,j+1,k,\ell} + \gamma q^{\| X \|}(W X)_{i,k-1}A_{i,j,k,\ell}\right).
\end{eqnarray*}
Note that to get the second line above from the first, we used
Lemma \ref{ij_reduction} to replace
$A_{i,j,k,\ell}$ by $q^{-2} A_{i+1,j+1,k,\ell}$ in the first two terms.  
We now 
use the recurrence for $A_{i,j,k,\ell}$ (see Definition \ref{def:matrices})
on the second and third terms, 
obtaining the following:
\small{\begin{multline*}
q \sum_{i,k} \Big( \frac{\beta}{\delta q^2}(W X D^{(r)})_{i+1,k}A_{i+1,j+1,k,\ell} - \frac{\alpha\beta}{\delta q^2}(W X)_{i+1,k-1}({\beta A_{i+1,j+1,k-1,\ell}}+{\delta q A_{i+1,j,k-1,\ell}} + {q A_{i+1,j+1,k-1,\ell-1}}) \Big. \\
\Big. + \gamma q^{\| X \|}(W X)_{i,k-1}(\beta A_{i,j,k-1,\ell}+\delta q A_{i,j-1,k-1,\ell} + q A_{i,j,k-1,\ell-1}) \Big).
\end{multline*}}
Using the definition of matrix multiplication we can 
get rid of the sums, 
obtaining:
\begin{multline*}
\frac{\beta}{\delta q}(W X D^{(r)} A)_{j+1,\ell} + \beta \gamma q^{\| X \|+1}(W X A)_{j,\ell}+\gamma\delta q^{\| X \|+2}(W X A)_{j-1,\ell} + \gamma q^{\| X \|+2}(W X A)_{j,\ell-1}\\
 - \frac{\alpha\beta^2}{\delta q}(W X A)_{j+1,\ell}-\alpha\beta(W X A)_{j,\ell} - \frac{\alpha\beta}{\delta}(W X A)_{j+1,\ell-1}.
\end{multline*}
Now we replace the $(WXD^{(r)}A)_{j+1,\ell}$ in the first term 
above using (4), obtaining:
\begin{multline*}
\frac{\beta}{\delta q}{\left( q(W X A D^{(r+1)})_{j+1,\ell}+\alpha\beta(WX A)_{j+1,\ell}-\gamma\delta q^{\| X \|+2}(WX A)_{j,\ell} \right)} + \beta \gamma q^{\| X \|+1}(W X A)_{j,\ell}\\
+\gamma\delta q^{\| X \|+2}(W X A)_{j-1,\ell} + \gamma q^{\| X \|+2}(W X A)_{j,\ell-1}
 - \frac{\alpha\beta^2}{\delta q}(W X A)_{j+1,\ell}-\alpha\beta(W X A)_{j,\ell} - \frac{\alpha\beta}{\delta}(W X A)_{j+1,\ell-1}.
\end{multline*}
The second and third terms above cancel with the seventh and fourth terms,
and we apply (2) to the $(WXAD^{(r+1)})_{j+1,\ell}$ in the first
term, to obtain:
\begin{multline*}
\frac{1}{\delta}{\left( \delta(W X A E^{(r+1)})_{j,\ell}+\alpha\beta(WX A)_{j+1,\ell-1}-\gamma\delta q^{\| X\|+2}(WX A)_{j,\ell-1} \right)} \\
+\gamma\delta q^{\| X \|+2}(W X A)_{j-1,\ell} + \gamma q^{\| X \|+2}(W X A)_{j,\ell-1}
-\alpha\beta(W X A)_{j,\ell} - \frac{\alpha\beta}{\delta}(W X A)_{j+1,\ell-1}.
\end{multline*}
After cancellation, we obtain, as desired, that
\[
q(WX E^{(r)} A)_{j,\ell} = (W X A E^{(r+1)})_{j,\ell} -\alpha\beta(W X A)_{j,\ell} +\gamma\delta q^{\| X \|+2}(W X A)_{j-1,\ell}. 
\]
\end{proof}

\begin{lem}\label{lem:2-1}
If equation (2) from Proposition \ref{prop:refined} holds, then so 
does equation (1).
\end{lem}

\begin{proof}
%
We use that $(WX D^{(r)} E^{(r)})_{j,\ell} = \sum_{i,k}(W X D^{(r)})_{i,k}E^{(r)}_{i,j,k,\ell}$, and apply (2) to $(W X D^{(r)})_{i,k}$, to obtain:
\[
(WX D^{(r)} E^{(r)})_{j,\ell} = \sum_{i,k}{\frac{1}{\beta}\left( \delta(W X E^{(r)})_{i-1,k}+\alpha\beta(WX)_{i,k-1}-\gamma\delta q^{\| X \|}(WX)_{i-1,k-1} \right)}E^{(r)}_{i,j,k,\ell}.
\]
In the first and third term on the right-hand side above,
we use Lemma \ref{ij_reduction} to replace $E_{i,j,k,\ell}^{(r)}$ by 
$q E_{i-1,j-1,k,\ell}^{(r)}$, and in the second term we apply the 
defining recurrence for $E$, obtaining:
\begin{multline*}
\sum_{i,k}\frac{\delta}{\beta}(W X E^{(r)})_{i-1,k}({q E^{(r)}_{i-1,j-1,k,\ell}}) + \sum_{i,k}\alpha (WX)_{i,k-1}{\left( \beta(D^{(r)}+E^{(r)})_{i,j,k-1,\ell}+q E^{(r)}_{i,j,k-1,\ell-1} \right)}\\
 - \sum_{i,k}\frac{\gamma\delta q^{\| X \|}}{\beta} (WX)_{i-1,k-1} 
({q E^{(r)}_{i-1,j-1,k,\ell}}).
\end{multline*}
We interpret the first two sums above as matrix multiplication,
and apply the defining recurrence for $E$ to the third term, obtaining:
\begin{multline*}
\frac{\delta q}{\beta}(W X E^{(r)} E^{(r)})_{j-1,\ell} + \alpha\beta (WX (D^{(r)}+E^{(r)}))_{j,\ell} +\alpha q (WX E^{(r)})_{j,\ell-1}\\
 - \sum_{i,k}\frac{\gamma\delta q^{\| X \|+1}}{\beta} (WX)_{i-1,k-1} {\left( \beta(D^{(r)}+E^{(r)})_{i-1,j-1,k-1,\ell}+q E^{(r)}_{i-1,j-1,k-1,\ell-1} \right)}.
\end{multline*}
We apply (2) to the $(WXE^{(r)} E^{(r)})_{j-1,\ell}$ in the first term,
and interpret the final sum as matrix multiplication, obtaining:
{\small
\begin{multline*}
\frac{\delta q}{\beta}{\frac{1}{\delta}\left( \beta(WX E^{(r)} D^{(r)})_{j,\ell} -\alpha\beta(WX E^{(r)})_{j,\ell-1}+\gamma\delta q^{\|X\|+1}(WX E^{(r)})_{j-1,\ell-1} \right)} + \alpha\beta (WX (D^{(r)}+E^{(r)}))_{j,\ell}\\
+\alpha q (WX E^{(r)})_{j,\ell-1}
 - \gamma\delta q^{\| X \|+1} (WX (D^{(r)}+E^{(r)}))_{j-1,\ell} -\frac{\gamma\delta q^{\| X \|+2}}{\beta} (WX E^{(r)})_{j-1,\ell-1}.
\end{multline*}
}

After cancellation, we obtain the desired quantity
$$q(WX E^{(r)} D^{(r)})_{j,\ell} + \alpha\beta (WX (D^{(r)}+E^{(r)}))_{j,\ell} - \gamma\delta q^{\| X \|+1} (WX (D^{(r)}+E^{(r)}))_{j-1,\ell}.
$$
\end{proof}

\subsection{Reducing (2) and (4) from Proposition \ref{prop:refined}
to Theorem \ref{THM24}}\label{sec:proof2}

In this section we show that equations (2) and (4) from
Proposition \ref{prop:refined}  are a consequence of  
Theorem \ref{THM24} below.
We then give a brief outline of the proof of Theorem \ref{THM24}.

\begin{defn} A vertical strip of \emph{type $F^{(t)}$} is a vertical strip with exactly $t$ short rhombi and at least two square tiles, where the bottom-most square tile contains a $\delta$, and there is a $\beta$ as the next Greek letter above the $\delta$ in the strip (necessarily in a square tile), 
as in 
the left side of Figure \ref{Ft}. 
Similarly, a vertical strip of \emph{type $G^{(t)}$} has a $\beta$ in its bottom-most square tile, and a $\delta$ as the next Greek letter above the $\beta$, 
as in the right side of Figure \ref{Ft}.
\end{defn}

\begin{defn}\label{def:FG}
Let $F^{(t)}_{i,j,k,\ell}$ be 
the weight generating function for all ways of adding a  new vertical strip 
of type $F^{(t)}$ to the left of a rhombic staircase tableau with
precisely 
$i$ horizontal $\delta$-strips and $k$ horizontal $\alpha/\gamma$-strips,
so as to obtain a new tableau with  $j$ horizontal $\delta$-strips 
and $\ell$ horizontal $\alpha/\gamma$-strips.
\end{defn} 

\begin{figure}[h]
\centering
\includegraphics[width=0.2\textwidth]{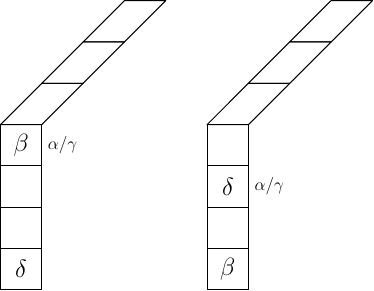}
\caption{Vertical strips of type  $F^{(3)}$ and $G^{(3)}$.}
\label{Ft}
\end{figure}

\begin{remark}
Since this section contains many long equations, we will 
try to simplify our notation somewhat.  
In particular, when $X$ and $Y$ are matrices, we will often write
$XY_{i,j,k,\ell}$ instead of $(XY)_{i,j,k,\ell}$, and similarly
$XYZ_{i,j,k,\ell}$ instead of $(XYZ)_{i,j,k,\ell}$.
\end{remark}

\begin{thm} \label{THM24}
Let $X$ be $(A,t)$-compatible, and set
 $r = |X|_A$.  Then we have that 
\small{
\begin{align}
\beta XD^{(r+t)}_{0,j,k,\ell}-\delta XE^{(r+t)}_{0,j-1,k,\ell}&=\alpha\beta X_{0,j,k,\ell-1}-\gamma\delta q^{2t+k+\| X \|}X_{0,j-1,k,\ell-1}+(1-q)F^{(t)} X_{0,j,k,\ell-1}, 
\label{EQ2} \\
 XD^{(r+t)}A_{0,j,k,\ell}-qXAD^{(r+t+1)}_{0,j,k,\ell}&=\alpha\beta XA_{0,j,k,\ell}-\gamma\delta q^{2t+k+2+\| X \|}XA_{0,j-1,k,\ell}+(1-q)F^{(t)} XA_{0,j,k,\ell}.
\label{EQ4}
\end{align}
}
\end{thm}

\begin{lemma}\label{lem:enough}
If Theorem \ref{THM24} holds, then so do 
equations (2) and (4) from
Proposition \ref{prop:refined}.  
\end{lemma}
\begin{proof}
Note that equation (2) (respectively, equation (4))
from Proposition \ref{prop:refined} is precisely
the case that $t=0$ and $k=0$ from \eqref{EQ2} (resp., \eqref{EQ4}).
Here we are using the fact that a vertical strip of type $F$
always contains a $\delta$, and hence
$(F^{(0)}X)_{0,j,0,\ell-1}= 
(F^{(0)}XA)_{0,j,0,\ell}=0$.
\end{proof}

The structure of the proof of Theorem \ref{THM24} is as follows.
In Section \ref{sec:FG}, we prove some recurrences for $F^{(t)}_{i,j,k,\ell}$
and $G^{(t)}_{i,j,k,\ell}$, and 
in Section \ref{xzero}, we prove Theorem \ref{THM24} when $|X|=0$.  In this case 
the proof of \eqref{EQ2} is very short while the proof of \eqref{EQ4} is quite long.
In Section \ref{efdf}, 
we prove  some useful identities (Proposition \ref{EFDF}), 
using Theorem \ref{THM24} in the case that $|X|=0$. 
Finally in Section \ref{induction}, 
we give the inductive proof of \eqref{EQ2} and 
 \eqref{EQ4}.  
This completes the proof of Theorem \ref{THM24}.

\begin{rem}
 Note that all the proofs use simple recurrences for $D,E$ and $F$. Even for $r=t=0$, they are much simpler than the proof given in
 \cite{CW-Duke2}, where generating functions and computer algebra were used. 
This gives the first human proof of the main result
of \cite{CW-Duke1}.
\end{rem}

\subsection{Recurrences for  $F^{(t)}_{i,j,k,\ell}$ and $G^{(t)}_{i,j,k,\ell}$}\label{sec:FG}

Recall from Definition \ref{def:FG} the definitions of 
$F^{(t)}_{i,j,k,\ell}$ and $G^{(t)}_{i,j,k,\ell}$.

\begin{lemma}\label{comb-lemmaF}
If $j>i \ge 0$ and $k\geq \ell\ge 0$, we have that 
\begin{align*}
F^{(t)}_{i,j,k,\ell}&=\delta E^{(t)}_{i,j-1,k-1,\ell}-\gamma\delta q^{2t+i+k-1} \one_{\{k=\ell,\ i=j-1\}}+F^{(t)}_{i,j,k-1,\ell-1} 
\label{recF}\\
G^{(t)}_{i,j,k,\ell}&=\beta D^{(t)}_{i,j,k-1,\ell}-\alpha\beta q^i \one_{\{k=\ell,\ i=j\}}+q G^{(t)}_{i,j,k-1,\ell-1}
\end{align*}
Otherwise we have that 
$F^{(t)}_{i,j,k,\ell}=G^{(t)}_{i,j,k,\ell} = 0.$
\end{lemma}

\begin{proof}
Note that if $j \leq i$ or $\ell \geq k$, then $F^{(t)}_{i,j,k,\ell}=0$. This is because any vertical strip of type $F^{(t)}$ contains at least one $\delta$ and at least one $\beta$, so adding a new vertical strip of this type to a rhombic staircase tableau necessarily increases the number of horizontal $\delta$-strips, and decreases 
the number of horizontal $\alpha/\gamma$-strips  by at least $1$.

In all other situations, we can assume that $k \geq 1$. 
Recall the combinatorial definition of 
$F^{(t)}_{i,j,k,\ell}$, and let $C$ be one of the vertical strips that it is
a sum over.  Let $T$ denote a representative tableau to which we will append $C$.
Let $B$ be the bottom-most square tile in $C$ which is part of 
an $\alpha/\gamma$-strip from $T$
 (such a square tile exists since $k \geq 1$).
If we fill $B$ with a $\beta$, 
then filling $B$ and the rest of the vertical strip $C$ is like
adding a new vertical strip with a $\beta$ in the bottom-most square tile to a 
tableau
with $i$ horizontal $\delta$-strips and $k-1$ horizontal $\alpha/\gamma$-strips.
This contributes
$\delta (E^{(t)}_{i,j-1,k-1,\ell}-\gamma q^{2t+i+k-1}\one_{\{k=\ell,\ i=j-1\}})$ 
to our generating polynomial.\footnote{The set of vertical strips with a $\beta$ in the bottom-most square tile is obtained by subtracting the set of vertical strips with a $\gamma$ in the bottom-most square tile from the set of vertical strips of type $E$. This explains the term containing $\gamma$ in the expression.}

On the other hand, if we leave $B$ empty,
then this square tile will get a weight $u=1$.
The generating function for filling the rest of the vertical strip $C$
is equal to 
$F^{(t)}_{i,j,k-1,\ell-1}$.  This completes the proof of the first recurrence.

The proof is the same for $G^{(t)}_{i,j,k,\ell}$. The base case where $G^{(t)}_{i,j,k,\ell}=0$ for $j \leq i$ or $\ell \geq k$ is literally the same. 
For the rest of the argument, the square tile $B$ can either contain a $\delta$, which
contributes
$\beta (D^{(t)}_{i,j-1,k-1,\ell}-\alpha q^i \one_{\{k=\ell,\ i=j-1\}})$, 
or a $q$, which contributes
$qG^{(t)}_{i,j,k-1,\ell-1}$ to our generating polynomial.
\end{proof}

\begin{lemma}\label{lem:FG}
For all $i,j,k$ and $\ell$,
\begin{equation*}
 F^{(t)}_{i,j,k,\ell}=G^{(t)}_{i,j,k,\ell}.
\end{equation*}
\end{lemma}

\begin{proof}
We give a bijective proof of the lemma.
Let $C$ be one of the vertical strips of type $F^{(t)}$ which 
$F^{(t)}_{i,j,k,\ell}$ is a sum over.  Let $T$ be a representative tableau
to which we are appending $C$.  
Let $m$ be the number of empty square tiles in $C$ which are:
 adjacent to
horizontal $\alpha/\gamma$-strips in $T$;
above the bottom-most $\delta$; and below the next
Greek letter (which is a $\beta$).  
And let $n$
be the number of empty square tiles which are:
adjacent to horizontal $\alpha/\gamma$-strips in $T$;
above the bottom-most $\beta$ in $C$; and below the next Greek letter in $C$. 
It is easy to see this strip has the same weight 
as a vertical strip of type $G^{(t)}$ in which the roles of $m$ and $n$
are reversed, 
see Figure \ref{bijFG}.
This map is clearly a bijection, so this proves the lemma.
\end{proof}

\begin{figure}[h]
\centering
\includegraphics[height=1.5in]{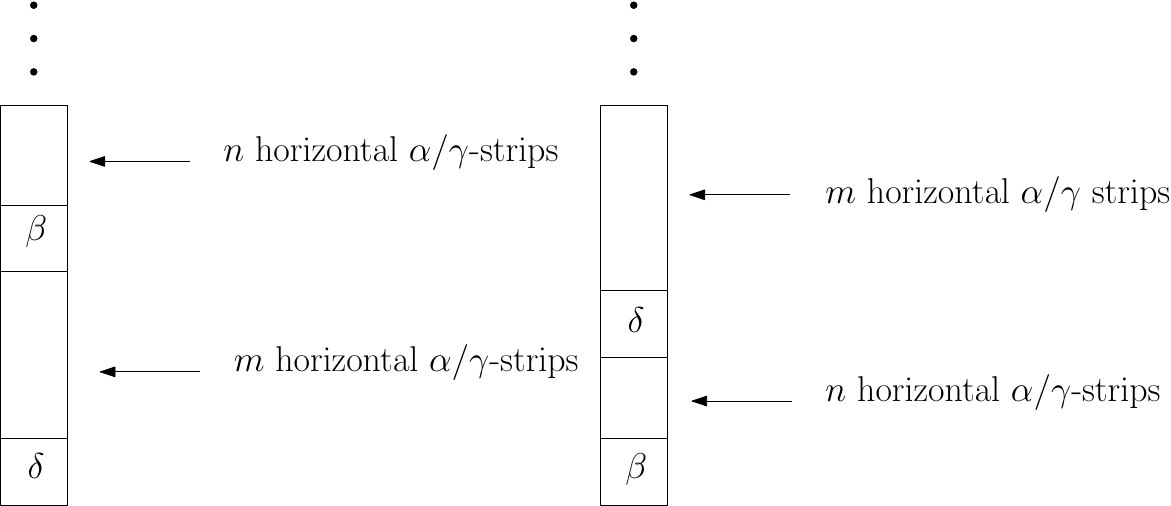}
\caption{The bijection between $F$ and $G$.}
\label{bijFG}
\end{figure}



\subsection{The proof of Theorem \ref{THM24} for $|X|=0$}
\label{xzero}

\subsubsection{The proof of \eqref{EQ2} when $|X|=0$.}

We need to show that 
\begin{equation} 
\beta D^{(t)}_{0,j,k,\ell}-\delta E^{(t)}_{0,j-1,k,\ell}=\alpha\beta \one_{\{j=0,\ k=\ell-1\}}
-\gamma\delta q^{2t+k}\one_{\{j=1,\ k=\ell-1\}}+(1-q)F^{(t)}_{0,j,k,\ell-1}.
\label{EQ2b}
\end{equation}
First consider the case that $k=0$.   
In this case \eqref{EQ2b} becomes
\begin{equation} 
\beta D^{(t)}_{0,j,0,\ell}-\delta E^{(t)}_{0,j-1,0,\ell}=\alpha\beta \one_{\{j=0,\ k=0,\ \ell=1\}}
-\gamma\delta q^{2t}\one_{\{j=1,\ k=0,\ \ell=1\}}+(1-q)F^{(t)}_{0,j,0,\ell-1}.
\label{EQ2bsecond}
\end{equation}
From Definition \ref{def:matrices}
we have that 
\[
D^{(t)}_{0,j,0,\ell}=\begin{cases}
\delta(q^t+(\alpha+\gamma q^t)[t]_q) & \mbox{if } j=1, \ell=0\\
\alpha & \mbox{if } j=0, \ell=1\\
0 & \mbox{otherwise}
\end{cases} 
\]
and also 
\[E^{(t)}_{0,j,0,\ell}=
\begin{cases}
\beta(q^t+(\alpha+\gamma q^t)[t]_q) & \mbox{if } j=0, \ell=0\\
\gamma q^{2t} & \mbox{if } j=0, \ell=1\\ 
0 & \mbox{otherwise.}
\end{cases}
\]
Using the fact that $F_{0,j,0,\ell-1}=0$, it is easy to check that we get
\eqref{EQ2bsecond}.

We now consider the case that $k>0$.
When $k>0$, we obtain the following recurrences
from 
Lemma \ref{comb-lemmaF}
and Lemma \ref{lem:FG}.
Note that we will 
omit the superscripts $(t)$ from $D$, $E$, and $F$ for readability.
\begin{eqnarray}
F_{i,j,k,\ell}&=&\delta E_{i,j-1,k-1,\ell}-
\gamma\delta q^{2t+i+k-1}
\one_{\{j=i+1,\ k=\ell\}}
+ F_{i,j,k-1,\ell-1}\label{recf1}\\
&=&\beta D_{i,j,k-1,\ell}-
\alpha\beta q^i
\one_{\{j=i,\ k=\ell\}}
+q F_{i,j,k-1,\ell-1}\label{recf2}
\end{eqnarray}

Using 
\eqref{recf1} and \eqref{recf2},
we get the desired expression
\begin{equation*}
\beta D_{0,j,k,\ell}-\delta E_{0,j-1,k,\ell}
 =F_{0,j,k,\ell-1}+\alpha\beta \one_{\{j=0,\ k=\ell-1\}}-qF_{0,j,k,\ell-1}-\gamma\delta q^{2t+k} \one_{\{j=1,k=\ell-1\}}.
\end{equation*}

\subsubsection{The proof of  \eqref{EQ4} when $|X|=0$.}

Note that Lemma \ref{base_caseF} directly implies
\eqref{EQ4} 
in the case that  $|X|=0$.
\begin{lemma}\label{base_caseF}
We have that 
\begin{align*}
D^{(t)} A_{0,j,k,\ell} - qA D^{(t+1)}_{0,j,k,\ell} &= 
A E^{(t+1)}_{0,j,k,\ell} - q E^{(t)} A_{0,j,k,\ell}\\
 &= \alpha\beta A_{0,j,k,\ell}-\gamma\delta q^{2t+k+2}A_{0,j-1,k,\ell}+(1-q) F^{(t)} A_{0,j,k,\ell}.
\end{align*}
\end{lemma}

\begin{proof}
Once again, we prove this by induction on $k$. If $k=0$, then 
we use the combinatorial interpretation of matrix products
from Theorem \ref{comb-interpret}(1) to obtain recurrences
for each of 
$AD^{(t+1)}_{0,j,0,\ell}$, 
$D^{(t)}A_{0,j,0,\ell}$, 
$AE^{(t+1)}_{0,j,0,\ell}$, and
$E^{(t)}A_{0,j,0,\ell}$. 
For example, the quantity $AD^{(t+1)}_{0,j,0,\ell}$ is the 
generating function for certain vertical strips consisting of 
one square box with an $\alpha$ or $\delta$ inside it, 
and with $t+1$ short rhombi above it, see Figure 
\ref{D_recurrence}.
\begin{figure}[h]
\centering
\includegraphics[height=1.5in]{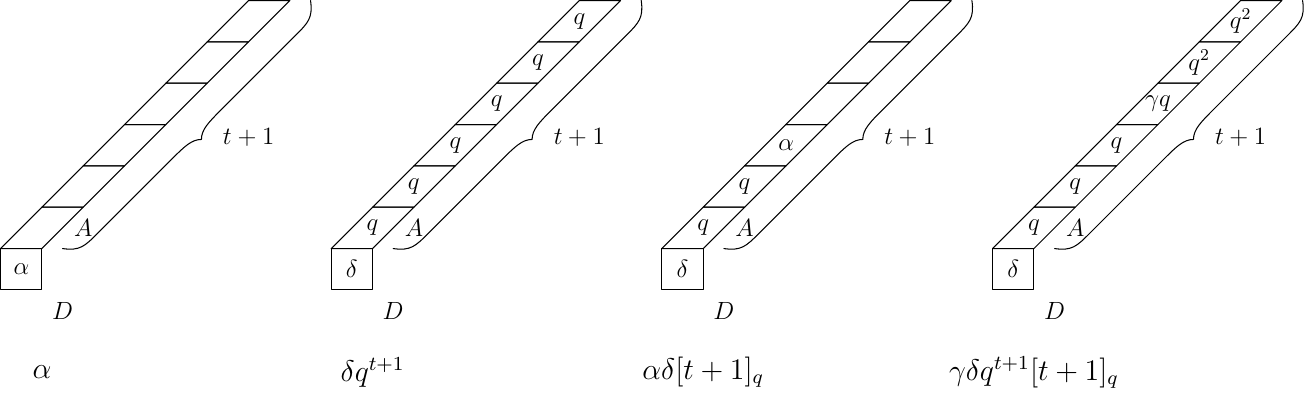}
\caption{The vertical strips contributing to $AD^{(t+1)}_{0,j,0,\ell}$}
\label{D_recurrence}
\end{figure}

Note that the third picture should be viewed as shorthand for 
$(t+1)$ different vertical strips, because there are $(t+1)$
choices for where to place the $\alpha$.  Similarly in the fourth
picture there are $(t+1)$ choices for where to place the $\gamma q$.
So we get:

\begin{align*}
AD^{(t+1)}_{0,j,0,\ell} &=
\begin{cases}
\alpha  & {\rm if} \ {j=0} \ {\rm and}\ {\ell=1}\\
\delta(q^{t+1}+(\alpha +\gamma q^{t+1})[t+1]_q) & {\rm if} \ {j=1}\  {\rm and}\ \ell =0\\
0 &{\rm otherwise}
\end{cases} 
\\
D^{(t)}A_{0,j,0,\ell}&=
\begin{cases}
\alpha q &  {\rm if} \ {j=0} \  {\rm and}\ {\ell=1}\\
\alpha \delta q + \delta q^2(q^t+(\alpha +\gamma q^{t})[t]_q) & {\rm if} \ {j=1}\  {\rm and}\ \ell=0\\
\alpha\beta   &  {\rm if} \ {j=0} \  {\rm and}\ {\ell=0}\\
0 & \qquad {\rm otherwise}
\end{cases}
\\
AE^{(t+1)}_{0,j,0,\ell}&=
\begin{cases}
\gamma q^{2t +2}  & {\rm if} \ {j=0} \  {\rm and}\ {\ell=1}\\
\beta(q^{t+1}+(\alpha +\gamma q^{t+1})[t+1]_q) & {\rm if} \ {j=0}\  {\rm and}\ \ell=0\\
0 &{\rm otherwise}
\end{cases}
\\
E^{(t)}A_{0,j,0,\ell}&=
\begin{cases}
\gamma q^{2t+1} & {\rm if} \ {j=0} \  {\rm and}\ {\ell=1}\\
\beta \gamma q^{2t} + \beta (q^t+(\alpha +\gamma q^{t})[t]_q) & {\rm if} \ {j=0}\  {\rm and}\ \ell=0\\
\gamma \delta q^{2t+1}   & {\rm if} \ {j=1} \  {\rm and}\ {\ell=0}\\
0 &{\rm otherwise}
\end{cases}
\end{align*}

In what follows, the superscripts $(t)$ and $(t+1)$ are implicit, 
but we omit them for readability. In each of the various cases
($j=0$ and $\ell=0$, etc.), we get
\begin{align*}
AE_{0,j,0,\ell}-qEA_{0,j,0,\ell}&=DA_{0,j,0,\ell}-qAD_{0,j,0,\ell}=
\alpha\beta\one_{\{j=0,\ell=0\}}-\gamma\delta q^{2t+2}\one_{\{j=1,\ell=0\}} \\
&= \alpha \beta A_{0,j,0,\ell}-\gamma \delta q^{2t+2} A_{0,j-1,0,\ell}
+(1-q) F^{(t)}A_{0,j,0,\ell} ,
\end{align*}
as desired (where in the last step we used the fact that 
$F^{(t)}A_{0,j,0,\ell}=0$).

Now we suppose that $k>0$.  We will prove the claim by induction. 
For $k>0$, we have from Definition \ref{def:matrices} that 
\begin{align}
D_{i,j,k,\ell}&=\delta(D+E)_{i,j-1,k-1,\ell}+D_{i,j,k-1,\ell-1}\label{recd} \text{ and } \\
A_{i,j,k,\ell}&=\beta A_{i,j,k-1,\ell}+\delta q A_{i,j-1,k-1,\ell}+q A_{i,j,k-1,\ell-1}.
\label{reca}
\end{align}
Using these, the definition of matrix multiplication, and 
Lemma \ref{ij_reduction}, we get:
{\small
\begin{eqnarray*}
 AD_{0,j,k,\ell}&=&\sum_{u,v}(\beta A_{0,u,k-1,v}+\delta q A_{0,u-1,k-1,v})D_{u,j,v,\ell}+q A_{0,u,k-1,v-1}(\delta (D+E)_{u,j-1,v-1,\ell}+D_{u,j,v-1,\ell-1})\\
 DA_{0,j,k,\ell}&=&\sum_{u,v} \delta(D+E)_{0,u-1,k-1,v}  q^2 A_{u-1,j-1,v,l} + D_{0,u,k-1,v-1} (\beta A_{u,j,v-1,\ell}+\delta q A_{u,j-1,v-1,\ell}+q A_{u,j,v-1,\ell-1}).
\end{eqnarray*}
}
This gives
\begin{align*}
AD_{0,j,k,\ell}&=\beta AD_{0,j,k-1,\ell}+\delta q^2 AD_{0,j-1,k-1,\ell}+\delta q A(D+E)_{0,j-1,k-1,\ell}+qAD_{0,j,k-1,\ell-1} \text{ and }\\
DA_{0,j,k,\ell}&=\beta DA_{0,j,k-1,\ell}+\delta q DA_{0,j-1,k-1,\ell}+\delta q^2 (D+E)A_{0,j-1,k-1,\ell}+qDA_{0,j,k-1,\ell-1}.
\end{align*}
We get
\begin{eqnarray*}
DA_{0,j,k,\ell}-qAD_{0,j,k,\ell}&=&
\delta q^2(DA_{0,j-1,k-1,\ell}-qAD_{0,j-1,k-1,\ell}-AE_{0,j-1,k-1,\ell}+qEA_{0,j-1,k-1,\ell}) \\
&&+\delta q(DA_{0,j-1,k-1,\ell}-qAD_{0,j-1,k-1,\ell})+\beta(DA_{0,j,k-1,\ell}-qAD_{0,j,k-1,\ell}\\
&&+ q(DA_{0,j,k-1,\ell-1}-qAD_{0,j,k-1,\ell-1})+\delta q^2(1-q)EA_{0,j-1,k-1,\ell}\\
\end{eqnarray*}
By the induction hypothesis, we suppose that Lemma \ref{base_caseF}
holds for $k-1$.  Substituting, we obtain:
\begin{eqnarray}
DA_{0,j,k,\ell}-qAD_{0,j,k,\ell}&=&\delta q\left(\alpha\beta A_{0,j-1,k-1,\ell}-\gamma \delta q^{2t+k+2-1}A_{0,j-2,k-1,\ell}+(1-q)FA_{0,j-1,k-1,\ell}\right) \nonumber \\
&& + \beta\left(\alpha\beta A_{0,j,k-1,\ell}-\gamma \delta q^{2t+k+2-1}A_{0,j-1,k-1,\ell}+(1-q)FA_{0,j,k-1,\ell}\right) \nonumber \\
&& + q\left(\alpha\beta A_{0,j,k-1,\ell-1}-\gamma \delta q^{2t+k+2-1}A_{0,j-1,k-1,\ell-1}+(1-q)FA_{0,j,k-1,\ell-1}\right) \nonumber \\
&&+\delta q^2(1-q)EA_{0,j-1,k-1,\ell} \nonumber \\
&=& \alpha\beta(\delta q  A_{0,j-1,k-1,\ell}+\beta  A_{0,j,k-1,\ell}+q A_{0,j,k-1,\ell-1}) \label{DAqAD1} \\
&&-\gamma \delta q^{2t+k+1} (\delta q  A_{0,j-2,k-1,\ell}+\beta  A_{0,j-1,k-1,\ell}+q A_{0,j-1,k-1,\ell-1}) \label{DAqAD2} \\
&&+(1-q)(\delta q FA_{0,j-1,k-1,\ell}+\beta FA_{0,j,k-1,\ell}+q FA_{0,j,k-1,\ell-1}+\delta q^2EA_{0,j-1,k-1,\ell}). \label{DAqAD3}
\end{eqnarray}

From  \eqref{reca}, it is clear that the quantity in 
 line  \eqref{DAqAD1} is equal to 
$\alpha\beta A_{0,j,k,\ell}$, and the quantity in line
\eqref{DAqAD2} is equal to 
$-\gamma \delta q^{2t+k+2-1}A_{0,j-1,k,\ell}$.

Using  \eqref{recf1}, \eqref{reca},  the definition of 
matrix multiplication, and Lemma \ref{ij_reduction}, we  get
\begin{equation*}
FA_{0,j,k,\ell} = \delta q FA_{0,j-1,k-1,\ell}+\beta FA_{0,j,k-1,\ell} + \delta q^2 EA_{0,j-1,k-1,\ell}-\gamma\delta q^{2t+k+1} A_{0,j-1,k,\ell}+q FA_{0,j,k-1,\ell-1},
\end{equation*}
which can be rewritten as
\begin{equation*}\label{FA-formula}
\delta q FA_{0,j-1,k-1,\ell}+\beta FA_{0,j,k-1,\ell}+q FA_{0,j,k-1,\ell-1}+\delta q^2 EA_{0,j-1,k-1,\ell}
=FA_{0,j,k,\ell}+\gamma\delta q^{2t+k+1}A_{0,j-1,k,\ell}.
\end{equation*} 
Using this to replace line \eqref{DAqAD3}, and then simplifying, 
we obtain the desired expression
\begin{eqnarray*}
DA_{0,j,k,\ell}-qAD_{0,j,k,\ell}&=& \alpha\beta A_{0,j,k,\ell}-\gamma \delta q^{2t+k+2}A_{0,j-1.k,\ell}+(1-q)FA_{0,j,k,\ell}.
\end{eqnarray*}

The proof that 
\begin{eqnarray*}
AE_{0,j,k,\ell}-qEA_{0,j,k,\ell}&=& \alpha\beta A_{0,j,k,\ell}-\gamma \delta q^{2t+k+2}A_{0,j-1.k,\ell}+(1-q)FA_{0,j,k,\ell}
\end{eqnarray*}
is completely analogous.
This concludes our proof of Lemma \ref{base_caseF}, which is 
the base case of $|X|=0$ in \eqref{EQ4}.
\end{proof}

\subsection{Commutator identities}
\label{efdf}

In this section we will prove the three ``commutator'' 
identities of
Proposition \ref{EFDF}.  
Actually we will prove \eqref{EF} and \eqref{AF} but not \eqref{DF}, because 
the proof of \eqref{DF} is analogous to that of \eqref{EF}.
Proposition \ref{EFDF} will be a useful tool
in the inductive proof of Theorem \ref{THM24}.

\begin{prop}\label{EFDF}
\begin{align}
(1-q)(E^{(t)}F^{(t)} - F^{(t)} E^{(t)})_{0,j,k,\ell} &= (1-q^j)\alpha\beta E^{(t)}_{0,j,k,\ell}+\gamma \delta q^{2t}(q^{\ell+j-1}-q^{k+1})E^{(t)}_{0,j-1,k,\ell},
\label{EF} \\
(1-q)(D^{(t)}F^{(t)} - F^{(t)}D^{(t)})_{0,j,k,\ell} &= (1-q^j)\alpha\beta D^{(t)}_{0,j,k,\ell}+\gamma \delta q^{2t}(q^{\ell+j-1}-q^{k+1})D^{(t)}_{0,j-1,k,\ell},
\label{DF} \\
(1-q)(AF^{(t+1)}-F^{(t)}A)_{0,j,k,\ell} &= (1-q^j)\alpha\beta A_{0,j,k,\ell}+\gamma\delta q^{2t}(q^{\ell+j+1}-q^{k+2})A_{0,j-1,k,\ell}.
\label{AF}
\end{align}
\end{prop}


Let us now prove Proposition \ref{EFDF} for $k=0$, and then for general $k$ by induction.

\subsubsection{The case $k=0$}


\begin{proof}

Note that from the combinatorial interpretation of $F$, we have 
that $F^{(t)}E^{(t)}_{i,j,0,\ell}=F^{(t)}D^{(t)}_{i,j,0,\ell}=
F^{(t)}A_{i,j,0,\ell}=AF^{(t)}_{i,j,0,\ell}=0$  for all $i, j$ and $\ell$.
(This is because when $k=0$, there are no horizontal $\alpha/\gamma$-strips
in the tableau we are building, so in particular there is no place to put
the $\beta$ and the $\delta$ in the vertical strip of type $F^{(t)}$.)

It is also easy to verify combinatorially (or directly from the matrices)
that 
\[
E^{(t)}F^{(t)}_{0,j,0,\ell}=\begin{cases}
\gamma\delta\beta q^{2t}(q^t+(\alpha+\gamma q^t)[t]_q) & {rm if}\ {j=1}\ {\rm and}\ {\ell=0}\\
0 & {\rm otherwise}.
\end{cases} 
\]

Now we use the base cases of the recurrences for $A$ and $E$: 
\[
A_{0,j,0,\ell}=\left\{
\begin{array}{ll}
1& {\rm if}\ j=0\ {\rm and}\ \ell=0\\
0 & {\rm otherwise.}\end{array}\right.
\]
\[E^{(t)}_{0,j,0,\ell}=\left\{
\begin{array}{ll}
\beta(q^t+(\alpha+\gamma q^t)[t]_q) & {\rm if}\ j=\ell=0\\
 \gamma q^{2t}& {\rm  if}\ j=0\ {\rm and}\ \ell=1\\
0 & {\rm otherwise.}\end{array}\right.
\]
A simple computation now yields the case $k=0$ of 
\eqref{EF} and \eqref{AF}. 
\end{proof}

\subsubsection{The case $k>0$}
We  prove \eqref{EF} by induction on $k$, supposing that 
\eqref{EF} and \eqref{DF} are true for $k-1$.
For simplicity, we again omit the superscripts $(t)$ of $E$ and $F$. 
Using 
Lemma \ref{comb-lemmaF}, and the fact that 
\begin{equation}
E_{i,j,k,\ell}=\beta(D+E)_{i,j,k-1,\ell}+qE_{i,j,k-1,\ell-1}, \label{rece1}
\end{equation}
we obtain
\begin{eqnarray*}
 EF_{0,j,k,\ell}&=&\sum_{u,v} E_{0,u,k,v} F_{u,j,v,\ell}\\
 &=&\sum_{u,v} \Big( \beta(D+E)_{0,u,k-1,v} F_{u,j,v,\ell}+qE_{0,u,k-1,v-1} F_{u,j,v,\ell} \Big) \\
 &=& \sum_{u,v} \Big( \beta(D+E)_{0,u,k-1,v} F_{u,j,v,\ell}+qE_{0,u,k-1,v-1}\delta E_{u,j-1,v-1,\ell} \Big. \\
 &&\Big. -qE_{0,u,k-1,v-1}\one_{\{j=u+1,\ v=\ell\}}\gamma\delta q^{2t+u+v-1}
 +qE_{0,u,k-1,v-1} F_{u,j,v-1,\ell-1} \Big)\\
&=& \beta(D+E)F_{0,j,k-1,\ell}+\delta q EE_{0,j-1,k-1,\ell} - \gamma\delta q^{2t+j+\ell-1}E_{0,j-1,k-1,\ell-1}+qEF_{0,j,k-1,\ell-1}.\\
\end{eqnarray*}

Similarly we get an expression for $FE$:
\begin{equation*}\label{recEF}
FE_{0,j,k,\ell} = \beta F(D+E)_{0,j,k-1,\ell}+\delta q EE_{0,j-1,k-1,\ell} - \gamma\delta q^{2t+k}E_{0,j-1,k,\ell}+qFE_{0,j,k-1,\ell-1}.
\end{equation*}


Putting the expressions for $EF$ and $FE$ together, we obtain
\begin{multline*}
(EF-FE)_{0,j,k,\ell} = \beta(EF-FE)_{0,j,k-1,\ell} + \beta(DF-FD)_{0,j,k-1,\ell} + q(EF-FE)_{0,j,k-1,\ell-1} \\
+\gamma\delta q^{2t+k}E_{0,j-1,k,\ell} - \gamma\delta q^{\ell+j+2t-1}E_{0,j-1,k-1,\ell-1}.
\end{multline*}

Using the induction hypothesis (that Proposition \ref{EFDF} holds for $k-1$),
we obtain:
\begin{multline*}
(1-q)(EF-FE)_{0,j,k,\ell} = \beta\Big( (1-q^j)
\underline{\underline{{\alpha\beta E_{0,j,k-1,\ell}}}}+\underline{{q^{2t}\gamma\delta(q^{\ell+j-1}-q^k)E_{0,j-1,k-1,\ell}}}\Big) \\
+ \beta\Big((1-q^j)\underline{\underline{{\alpha\beta D_{0,j,k-1,\ell}}}}+\underline{{q^{2t}\gamma\delta(q^{\ell+j-1}-q^k)D_{0,j-1,k-1,\ell}}}\Big) \\
+ q\Big((1-q^j)\underline{\underline{{\alpha\beta E_{0,j,k-1,\ell-1}}}}+\underline{{q^{2t}\gamma\delta(q^{\ell+j-2}-q^k)E_{0,j-1,k-1,\ell-1}}}\Big) \\
+(1-q)\gamma\delta q^{2t+k}E_{0,j-1,k,\ell} - \underline{(1-q)\gamma\delta q^{\ell+j+2t-1}E_{0,j-1,k-1,\ell-1}}. 
\end{multline*}

Now we group the terms according to the underlines above, 
and use the defining recurrences for 
$D_{i,j,k,\ell}$ and $E_{i,j,k,\ell}$ given in \eqref{reca} and \eqref{rece1}.  This gives
\begin{eqnarray*}
(1-q)(EF-FE)_{0,j,k,\ell} &=& \underline{\underline{{(1-q^j)\alpha\beta E_{0,j,k,\ell}}}} + 
\underline{q^{2t+\ell+j-1} \gamma \delta E_{0,j-1,k,\ell}}
- \underline{q^{2t+k} \gamma \delta E_{0,j-1,k,\ell}}\\
&&
+(1-q) \gamma \delta q^{2t+k} E_{0,j-1,k,\ell}\\
  &=& (1-q^j)\alpha\beta E_{0,j,k,\ell} +  q^{2t+\ell+j-1} \gamma \delta E_{0,j-1,k,\ell} -  q^{2t+k+1} \gamma \delta E_{0,j-1,k,\ell},
\end{eqnarray*}
which completes the proof of 
\eqref{EF}.

We now prove the identity for $AF-FA$. 
%
Using the recurrences of  \eqref{reca} and \eqref{recf2}, 
we obtain the following:
\begin{align*}
AF_{0,j,k,\ell} &= \delta q^2 AF_{0,j-1,k-1,\ell}+\beta AF_{0,j,k-1,\ell}+q^2 AF_{0,j,k-1,\ell-1} + \beta q AD_{0,j,k-1,\ell}-\alpha\beta q^{j+1} A_{0,j,k-1,\ell-1},\\
FA_{0,j,k,\ell} &= \delta q^2 FA_{0,j-1,k-1,\ell}+\beta q FA_{0,j,k-1,\ell}+q^2 FA_{0,j,k-1,\ell-1} + \beta DA_{0,j,k-1,\ell}-\alpha\beta A_{0,j,k,\ell}.
\end{align*}
We have that 
\begin{eqnarray*}
(1-q)(AF-FA)_{0,j,k,\ell} &=&
 \delta q^2((1-q)(AF - FA)_{0,j-1,k-1,\ell}) 
+\beta(1-q)AF_{0,j,k-1,\ell}\\ 
&-& \beta q(1-q)FA_{0,j,k-1,\ell}+q^2((1-q)(AF-FA)_{0,j,k-1,\ell-1})\\
&-& (1-q)\beta (DA-q AD)_{0,j,k-1,\ell} + (1-q)\alpha \beta (A_{0,j,k,\ell} - q^{j+1} A_{0,j,k-1,\ell-1}).
\end{eqnarray*}
We now use the induction hypothesis along with Lemma \ref{base_caseF}.
We then find
that $(1-q)(AF-FA)_{0,j,k,\ell}$ is equal to:
%
\begin{eqnarray*}
&& \delta q^2 \left((1-q^{j-1})\alpha\beta A_{0,j-1,k-1,\ell}+\gamma\delta q^{2t}(q^{\ell+j}-q^{k+1})A_{0,j-2,k-1,\ell} \right)\\
&+&\beta(1-q)AF_{0,j,k-1,\ell} - \underline{\beta q(1-q)FA_{0,j,k-1,\ell}}\\
&+&q^2 \left( (1-q^j)\alpha\beta A_{0,j,k-1,\ell-1}+\gamma\delta q^{2t}(q^{\ell+j}-q^{k+1})A_{0,j-1,k-1,\ell-1} \right)\\
&-& \underline{(1-q)\beta} \left(\alpha\beta A_{0,j,k-1,\ell}-\gamma\delta q^{2t+k+1}A_{0,j-1,k-1,\ell}+\underline{(1-q) (F A)_{0,j,k-1,\ell}}\right)\\
 &+& (1-q)\alpha\beta (A_{0,j,k,\ell} - q^{j+1} A_{0,j,k-1,\ell-1}).
\end{eqnarray*}
We group the terms according to the underlines, obtaining:
\begin{eqnarray*}
&&\delta q^2 \left((1-q^{j-1})\alpha\beta A_{0,j-1,k-1,\ell}+\gamma\delta q^{2t}(q^{\ell+j}-q^{k+1})A_{0,j-2,k-1,\ell} \right)\\
&+&\beta((1-q)(AF-FA)_{0,j,k-1,\ell}) \\
&+&q^2 \left( (1-q^j)\alpha\beta A_{0,j,k-1,\ell-1}+\gamma\delta q^{2t}(q^{\ell+j}-q^{k+1})A_{0,j-1,k-1,\ell-1} \right)\\
&-& (1-q)\beta \left(\alpha\beta A_{0,j,k-1,\ell}-\gamma\delta q^{2t+k+1}A_{0,j-1,k-1,\ell}\right)\\
 &+& (1-q)\alpha\beta (A_{0,j,k,\ell} - q^{j+1} A_{0,j,k-1,\ell-1}).
\end{eqnarray*}
We use again the induction hypothesis to get:
\begin{eqnarray*}
&&\delta q^2 \left((1-q^{j-1})\alpha\beta A_{0,j-1,k-1,\ell}+\gamma\delta q^{2t}(q^{\ell+j}-q^{k+1})A_{0,j-2,k-1,\ell} \right)\\
&+&\beta\left((1-q^j){\alpha\beta A_{0,j,k-1,\ell}}+\gamma\delta q^{2t}(q^{\ell+j+1}-{q^{k+1})A_{0,j-1,k-1,\ell}}\right) \\
&+&q^2 \left( (1-q^j){\alpha\beta A_{0,j,k-1,\ell-1}}+\gamma\delta q^{2t}(q^{\ell+j}-q^{k+1})A_{0,j-1,k-1,\ell-1} \right)\\
&-& (1-q)\beta \left({\alpha\beta A_{0,j,k-1,\ell}}-{\gamma\delta q^{2t+k+1}A_{0,j-1,k-1,\ell}}\right)\\
 &+& (1-q)(\alpha\beta A_{0,j,k,\ell} -{\alpha\beta q^{j+1} A_{0,j,k-1,\ell-1}}).
\end{eqnarray*}
%
%
We combine lines 2 and 4 and then lines 3 and 5 to get:
\begin{eqnarray*}
&&\delta q^2 \left((1-q^{j-1})\alpha\beta \underline{A_{0,j-1,k-1,\ell}}+\gamma\delta q^{2t}(q^{\ell+j}-q^{k+1})\underline{\underline{A_{0,j-2,k-1,\ell}}} \right)\\
&+&\beta q((1-q^{j-1})\alpha\beta \underline{A_{0,j,k-1,\ell}}+\gamma\delta q^{2t}(q^{\ell+j}-q^{k+1})\underline{\underline{A_{0,j-1,k-1,\ell}}}) \\
&+&q^2 \left( (1-q^{j-1})\alpha\beta \underline{A_{0,j,k-1,\ell-1}}+\gamma\delta q^{2t}(q^{\ell+j}-q^{k+1})\underline{\underline{A_{0,j-1,k-1,\ell-1}}} \right)
 + (1-q)\alpha\beta A_{0,j,k,\ell}
\end{eqnarray*}
We apply  \eqref{reca} to the underlined terms, and obtain:
\begin{eqnarray*}
&&\alpha\beta q (1-q^{j-1}) \underline{A_{0,j,k,\ell}} +\gamma\delta q^{2t+1}(q^{\ell+j}-q^{k+1})\underline{\underline{A_{0,j-1,k,\ell}}} 
 + (1-q)\alpha\beta A_{0,j,k,\ell}\\
 &=& \alpha\beta (1-q^{j}) A_{0,j,k,\ell} +\gamma\delta q^{2t+1}(q^{\ell+j}-q^{k+1})A_{0,j-1,k,\ell},
\end{eqnarray*}
as desired. This completes the proof of Proposition \ref{EFDF}.

\subsection{The inductive proof of Theorem \ref{THM24}}
\label{induction}

We now prove Theorem \ref{THM24} by induction on $|X|$.
We already verified the base case ($|X|=0$) in Section
\ref{xzero}.

\subsubsection{The inductive proof of  \eqref{EQ2}}


We need to prove that for $X$ an $(A,t)$-compatible word, 
we have \eqref{Eq2} below.
\begin{equation}
\beta XD_{0,j,k,\ell} - \delta XE _{0,j-1,k,\ell} = \alpha\beta X_{0,j,k,\ell-1}-\gamma\delta q^{2t+k+\| X\|}X_{0,j-1,k,\ell-1}+(1-q) FX_{0,j,k,\ell-1}.
\label{Eq2}
\end{equation}
Note that by Lemma \ref{ij_reduction}, \eqref{Eq2} is equivalent to 
\begin{equation}
\beta XD_{i,j,k,\ell} - \delta XE _{i,j-1,k,\ell} = \alpha\beta q^i X_{i,j,k,\ell-1}-\gamma\delta q^{2t+k+i+\| X\|}X_{i,j-1,k,\ell-1}+(1-q) FX_{i,j,k,\ell-1}.
\label{Eq2-general}
\end{equation}
For our proof, we will show that if \eqref{Eq2} holds for 
$X=Y$, where $Y$ is $(A,t)$-compatible, then \eqref{Eq2} also holds for
$X=DY$, $X=EY$, and $X=AY$. 

We start by showing that after plugging $X=DY$ into \eqref{Eq2}, the 
left-hand side (LHS) will match the right-hand side (RHS).
On the LHS we have:
\begin{multline*}
\beta(DYD)_{0,j,k,\ell} - \delta(DYE)_{0,j-1,k,\ell}\\
\begin{aligned}
&= \sum_{u,v}D_{0,u,k,v} 
\left(\beta YD_{u,j,v,\ell} - \delta YE_{u,j-1,v,\ell}\right)\\
&=\sum_{u,v} D_{0,u,k,v} \left(\alpha\beta q^u Y_{u,j,v,\ell-1}
-\gamma\delta q^{2t+v+u+\| Y\|}Y_{u,j-1,v,\ell-1}+(1-q)FY_{u,j,v,\ell-1}
\right)\\
&= \sum_{u,v} \alpha\beta q^u D_{0,u,k,v}Y_{u,j,v,\ell-1}
-\sum_{u,v}\gamma\delta q^{2t+v+u}
D_{0,u,k,v} Y_{u+1,j,v,\ell-1}
+\sum_{u,v}(1-q)D_{0,u,k,v} FY_{u,j,v,\ell-1}.
\end{aligned}
\end{multline*}
Here we used the induction hypothesis 
for \eqref{Eq2-general}
to obtain the third line from the second, and we 
used Lemma \ref{ij_reduction} in the second term
to obtain the fourth line from the third.
Now in the fourth line, 
we change our index of summation in the second sum, replacing each $u$ by $u-1$.
Also note that $\sum_{u,v} D_{0,u,k,v} FY_{u,j,v,\ell-1} = 
\sum_{u,v} DF_{0,u,k,v} Y_{u,j,v,\ell-1}$.
Thus we obtain that the LHS equals
\[
\sum_{u,v} \left(\alpha\beta q^u D_{0,u,k,v} - \gamma\delta q^{2t+v+u-1}D_{0,u-1,k,v} +(1-q)DF_{0,u,k,v}\right) Y_{u,j,v,\ell-1}.
\]
The RHS for $X=DY$ is obtained in a similar way:
\begin{multline*}
\alpha\beta DY_{0,j,k,\ell-1} - \gamma\delta q^{2t+k+\|DY\|}DY_{0,j-1,k,\ell-1}+(1-q)FDY_{0,j,k,\ell-1}\\
\begin{aligned}
&= \sum_{u,v} D_{0,u,k,v} \left(\alpha\beta Y_{u,j,v,\ell-1} - \gamma\delta q^{2t+k+\|DY\|}Y_{u,j-1,v,\ell-1}\right) 
+\sum_{u,v}(1-q)FD_{0,u,k,v}Y_{u,j,v,\ell-1}\\
&= \sum_{u,v} \left(\alpha\beta D_{0,u,k,v} - \gamma\delta q^{2t+k+\| Y\|+1-\| Y\|}D_{0,u-1,k,v} +(1-q)FD_{0,u,k,v}\right)Y_{u,j,v,\ell-1}.
\end{aligned}
\end{multline*}

Equation \eqref{DF}  of Proposition \ref{EFDF}  implies that
\begin{multline*}
\alpha\beta q^u D_{0,u,k,v} - \gamma\delta q^{2t+v+u-1}D_{0,u-1,k,v} +(1-q)DF_{0,u,k,v} \\
= \alpha\beta D_{0,u,k,v} - \gamma\delta q^{2t+k+1}D_{0,u-1,k,v} +(1-q)FD_{0,u,k,v},
\end{multline*}
which means the LHS equals the RHS of \eqref{Eq2} for $X=DY$.

If $X=EY$, the proof is the same. The induction is equivalent to proving that
\begin{multline*}
\alpha\beta q^u E_{0,u,k,v} - \gamma\delta q^{2t+v+u-1}E_{0,u-1,k,v} +(1-q)EF_{0,u,k,v} \\
= \alpha\beta E_{0,u,k,v} - \gamma\delta q^{2t+k+1}E_{0,u-1,k,v} +(1-q)FE_{0,u,k,v},
\end{multline*}
which is implied by  \eqref{EF}.
Consequently the LHS equals the RHS of \eqref{Eq2} for $X=EY$.

Finally we consider the case
 $X=AY$, where $X$ is an $(A,t)$-compatible word (so $Y$ is 
an $(A,t+1)$-compatible word).  The proof is nearly the same
as before.  We have that the LHS of \eqref{Eq2} equals:

\begin{multline*}
\beta AYD_{0,j,k,\ell} - \delta AYE_{0,j-1,k,\ell}\\
\begin{aligned}
&= \sum_{u,v}A_{0,u,k,v} (\beta YD_{u,j,v,\ell} - \delta YE_{u,j-1,v,\ell})\\
&=\sum_{u,v} A_{0,u,k,v} (\alpha\beta q^u Y_{u,j,v,\ell-1}
-\gamma\delta q^{2(t+1)+v+u+\| Y\|}Y_{u,j-1,v,\ell-1}+(1-q)FY_{u,j,v,\ell-1})\\
&=\sum_{u,v} \alpha\beta q^u A_{0,u,k,v}Y_{u,j,v,\ell-1}
-\sum_{u,v}\gamma\delta q^{2t+2+v+u}A_{0,u,k,v} Y_{u+1,j,v,\ell-1}
+\sum_{u,v}(1-q) AF_{0,u,k,v} Y_{u,j,v,\ell-1}.
\end{aligned}
\end{multline*}
In the second sum on the last line, we now change the index of summation,
replacing each $u$ by $u-1$.
Thus we obtain that the LHS of \eqref{Eq2} equals
\[
\sum_{u,v} (\alpha\beta q^u A_{0,u,k,v} - \gamma\delta q^{2t+v+u+1}A_{0,u-1,k,v} +(1-q)AF_{0,u,k,v}) Y_{u,j,v,\ell-1}.
\]
We obtain the RHS of \eqref{Eq2} for $X=AY$ in the same way:
\begin{multline*}
\alpha\beta AY_{0,j,k,\ell-1} - \gamma\delta q^{2t+k+\|AY\|} AY_{0,j-1,k,\ell-1}+(1-q) FAY_{0,j,k,\ell-1}\\
\begin{aligned}
&= \sum_{u,v} A_{0,u,k,v} \left(\alpha\beta Y_{u,j,v,\ell-1} - \gamma\delta q^{2t+k+\|AY\|}Y_{u,j-1,v,\ell-1}\right) +\sum_{u,v}(1-q) FA_{0,u,k,v}Y_{u,j,v,\ell-1}\\
&= \sum_{u,v} \left(\alpha\beta A_{0,u,k,v} - \gamma\delta q^{2t+k+\| Y\|+2-\| Y\|}A_{0,u-1,k,v} +(1-q) FA_{0,u,k,v}\right)Y_{u,j,v,\ell-1}.
\end{aligned}
\end{multline*}
Equation \eqref{AF}  of Proposition \ref{EFDF} implies that
\begin{multline*}
\alpha\beta q^u A_{0,u,k,v} - \gamma\delta q^{2t+v+u+1}A_{0,u-1,k,v} +(1-q) AF_{0,u,k,v} \\
= \alpha\beta A_{0,u,k,v} - \gamma\delta q^{2t+k+2}A_{0,u-1,k,v} +(1-q)FA_{0,u,k,v},
\end{multline*}
which means the LHS equals the RHS of \eqref{Eq2} for $X=AY$.

This completes the proof of   \eqref{EQ2}.

\subsubsection{The inductive proof of \eqref{EQ4}}

We want to prove that for any $X$ that is $(A,t)$-compatible,
\begin{equation}
\label{Eq4}
XDA_{0,j,k,\ell}-qXAD_{0,j,k,\ell}=\alpha\beta XA_{0,j,k,\ell}-\gamma\delta q^{2t+k+2+\| X\|}XA_{0,j-1,k,\ell}+(1-q)FXA_{0,j,k,l}.
\end{equation}
Note that by Lemma \ref{ij_reduction}, \eqref{Eq4} is equivalent to 
\begin{equation}
XDA_{i,j,k,\ell}-qXAD_{i,j,k,\ell}=\alpha\beta q^i XA_{i,j,k,\ell}-\gamma\delta q^{2t+i+k+2+\| X\|}XA_{i,j-1,k,\ell}+(1-q)FXA_{i,j,k,l}.
\label{Eq4-general}
\end{equation}

We know  that \eqref{Eq4} holds for $|X|=0$ by Lemma \ref{base_caseF}. 
For the inductive step, we will show that if \eqref{Eq4} holds for 
$X=Y$,  then \eqref{Eq4} also holds for
$X=AY$, $X=DY$, and $X=EY$. 


When  $X=AY$, the LHS of \eqref{Eq4} becomes
\begin{eqnarray*}
AYDA_{0,j,k,\ell}-qAYAD_{0,j,k,\ell}
&=&\sum_{u,v} A_{0,u,k,v} (YDA_{u,j,v,\ell}-qYAD_{u,j,v,\ell})\\
&=&\sum_{u,v} A_{0,u,k,v} q^{u(\| Y\|+3)}(YDA_{0,j-u,v,\ell}-qYAD_{0,j-u,v,\ell})\\
\end{eqnarray*}
where we applied Lemma \ref{ij_reduction} to $YDA$ and $YAD$ to get the 
second equality from the first.  
Note that since $X=AY$ is $(A,t)$-compatible, $Y$ must be 
$(A,t+1)$-compatible.
We now use the induction hypothesis 
\eqref{Eq4-general} to obtain that when $X=AY$, 
the LHS of \eqref{Eq4} 
equals
%
\begin{align*}
&=\sum_{u,v} A_{0,u,k,v} \left(\alpha \beta q^u YA_{u,j,v,\ell}
-\gamma \delta q^{2(t+1)+u+v+2+\|Y\|} YA_{u,j-1,v,\ell} 
+(1-q) FYA_{u,j,v,\ell} \right)\\
&=\sum_{u,v} \left( \alpha\beta q^{u} A_{0,u,k,v} YA_{u,j,v,\ell}-\gamma\delta q^{2t+4+u+v+\| Y\|} A_{0,u,k,v} YA_{u,j-1,v,\ell}
+(1-q)AF_{0,u,k,v}YA_{u,j,v,\ell} \right) \nonumber\\
&=\sum_{u,v} \big( \alpha\beta q^{u} A_{0,u,k,v} YA_{u,j,v,\ell}-\gamma\delta q^{2t+4+u+v+\| Y\|-\|YA\|} A_{0,u,k,v} YA_{u+1,j,v,\ell} \big. 
\label{eqAYDA2}\\
& \big. \qquad\qquad\qquad +(1-q)AF_{0,u,k,v}YA_{u,j,v,\ell} \big) \nonumber\\
&=\sum_{u,v} \left(\alpha \beta q^u A_{0,u,k,v}-\gamma\delta q^{2t+u+v+1} A_{0,u-1,k,v}
+(1-q)AF_{0,u,k,v}\right)YA_{u,j,v,\ell}. \nonumber
\end{align*}
To go from the first to the second line above, we used the fact that 
$\sum_{u,v} A_{0,u,k,v} FYA_{u,j,v,\ell} = \sum_{u,v} AF_{0,u,k,v} YA_{u,j,v,\ell}$.  
To go from the second to the third line, we used Lemma \ref{ij_reduction}
in the second sum.
And to go from the third to the fourth line, we changed the 
index of summation in the second term, replacing each $u$ by $u-1$.


Similarly, when $X=AY$, the RHS of  \eqref{Eq4} is
\begin{multline*}
\alpha\beta AYA_{0,j,k,\ell}-\gamma\delta q^{2t+k+2+\| Y\|+2}AYA_{0,j-1,k,\ell}+(1-q)FAYA_{0,j,k,l}\\
\begin{aligned}
&= \sum_{u,v} \big(\alpha \beta A_{0,u,k,v}YA_{u,j,v,\ell}-\gamma\delta q^{2t+k+4+\| Y\|}A_{0,u,k,v}YA_{u,j-1,v,\ell}+(1-q)FA_{0,u,k,v}YA_{u,j,v,\ell}\big)\\
&= \sum_{u,v} \big( \alpha \beta A_{0,u,k,v}YA_{u,j,v,\ell}-\gamma\delta q^{2t+k+4+\| Y\|-\| YA\|}A_{0,u,k,v}YA_{u+1,j,v,\ell}+(1-q)FA_{0,u,k,v}YA_{u,j,v,\ell}\big)\\
&=\sum_{u,v} \big(\alpha \beta A_{0,u,k,v}-\gamma\delta q^{2t+k+2} A_{0,u-1,k,v}
+(1-q)FA_{0,u,k,v}\big)YA_{u,j,v,\ell}.
\end{aligned}
\end{multline*}
Consequently, if we can show that
{\small \[
\alpha \beta q^u A_{0,u,k,v}-\gamma\delta q^{2t+u+v+1} A_{0,u-1,k,v}
+(1-q)AF_{0,u,k,v}=
\alpha \beta A_{0,u,k,v}-\gamma\delta q^{2t+k+2} A_{0,u-1,k,v}
+(1-q)FA_{0,u,k,v},
\]
}
 then the LHS and RHS of \eqref{Eq4} will be equal. 
But the equation above is exactly  \eqref{AF} from Proposition \ref{EFDF}. Therefore if  \eqref{Eq4} is true for $X=Y$, it is also true for $X=AY$.

Let us now prove  \eqref{Eq4} for $X=DY$. The LHS of \eqref{Eq4} is
\begin{equation*}
DYDA_{0,j,k,\ell}-qDYAD_{0,j,k,\ell}
=\sum_{u,v} D_{0,u,k,v} (YDA_{u,j,v,\ell}-qYAD_{u,j,v,\ell}).
\end{equation*}
Using the induction hypothesis and applying \eqref{Eq4-general}, we get
\begin{align*}
&=\sum_{u,v} D_{0,u,k,v} \big(\alpha\beta q^{u} YA_{u,j,v,\ell}
-\gamma\delta q^{2t+v+u+2+\|Y\|} YA_{u,j-1,v,\ell}
 +(1-q)FYA_{u,j,v,\ell}\big)\\
&=\sum_{u,v} \big(\alpha\beta q^u D_{0,u,k,v} YA_{u,j,v,\ell}-\gamma\delta 
q^{2t+u+v} D_{0,u,k,v} YA_{u+1,j,v,\ell} 
+(1-q)DF_{0,u,k,v}YA_{u,j,v,\ell} \big)\\
&=\sum_{u,v} \big(\alpha \beta q^u D_{0,u,k,v}-\gamma\delta q^{2t+u-1+v} D_{0,u-1,k,v}
+(1-q)DF_{0,u,k,v}\big) YA_{u,j,v,\ell}.
\end{align*}

When $X=DY$, the RHS of \eqref{Eq4} is 
\begin{multline*}
\alpha\beta DYA_{0,j,k,\ell}-\gamma\delta q^{2t+k+2+\| Y\|+1}DYA_{0,j-1,k,\ell}+(1-q)FDYA_{0,j,k,\ell}\\
\begin{aligned}
&= \sum_{u,v} \big( \alpha \beta D_{0,u,k,v}YA_{u,j,v,\ell}-\gamma\delta q^{2t+k+3+\| Y\|}D_{0,u,k,v}YA_{u,j-1,v,\ell}+(1-q)FD_{0,u,k,v}YA_{u,j,v,\ell}\big)\\
&= \sum_{u,v} \big( \alpha \beta D_{0,u,k,v}YA_{u,j,v,\ell}-\gamma\delta q^{2t+k+3+\|Y\|-\|YA\|}D_{0,u,k,v}YA_{u+1,j,v,\ell}+(1-q)FD_{0,u,k,v}YA_{u,j,v,\ell}\big)\\
&=\sum_{u,v} \big(\alpha \beta D_{0,u,k,v}-\gamma\delta q^{2t+k+1} D_{0,u-1,k,v}
+(1-q)FD_{0,u,k,v}\big)YA_{u,j,v,\ell}.
\end{aligned}
\end{multline*}

Thus, if we can show that 
{\small
\[\alpha \beta q^u D_{0,u,k,v}-\gamma\delta q^{2t+u+v-1} D_{0,u-1,k,v}
+(1-q)DF_{0,u,k,v}=
\alpha \beta D_{0,u,k,v}-\gamma\delta q^{2t+k+1} D_{0,u-1,k,v}
+(1-q)FD_{0,u,k,v},
\]
}
then it follows that for $X=DY$, the LHS of \eqref{Eq4} equals the RHS. 
But the above equation is exactly \eqref{DF} from Proposition \ref{EFDF}. Therefore, if \eqref{Eq4} is true for $X=Y$, it is also true for $X=DY$.

The proof for $X=EY$ follows the same steps as the proof for $X=DY$. This completes the proof of Theorem \ref{THM24}.

\section{The proof of our formula for Koornwinder moments
}\label{sec:momentproof}

Our goal in this section is to prove 
Theorem \ref{thm:moment}.  Note that it is a direct consequence of 
Theorem \ref{thm:CW-K} 
and Lemma \ref{lem:2partition} 
below.

\begin{definition}
Let $\DD$, $\EE$, and $\AAA$ be matrices, and 
 $\langle \WW|$ and $|\VV\rangle$ be vectors, such that the following relations hold:
\begin{enumerate}
\item $\langle \WW| (\alpha \EE-\gamma \DD) = \langle \WW|$ \label{eq:m1}
\item $(\beta \DD-\delta \EE)|\VV\rangle = |\VV\rangle$
\item  $\DD\EE-q\EE\DD = \DD+\EE$
\item $\DD\AAA = q\AAA\DD+\AAA$
\item $\AAA\EE = q\EE\AAA+\AAA$. \label{eq:m5}
\end{enumerate}
(By work of \cite{USW} and \cite{Uchiyama} we know that such matrices and vectors exist;
see also \cite[Section 3.4]{CW-Koornwinder}.)
We then define $\mathcal{Z}_{N,r}(\xi) = [y^r] \frac{\langle \WW| (\xi \DD+\EE+y\AAA)^N |V \rangle}{\langle \WW| \AAA^r|\VV \rangle}$.
\end{definition}

\begin{remark}
Note that if we have $\DD, \EE, \AAA, \langle \WW|$, and 
$|\VV \rangle$ as above, we can set 
$D^{(t)} = \DD$  
and $E^{(t)} = \EE$ for all $t$, 
and set $A = \AAA$, $\langle W| = \langle \WW|$, and 
$|V\rangle = |\VV\rangle$.  Then we have a solution to the 
relations (I) through (V) of Theorem \ref{thm:MA}, with 
$\lambda_n = 1$ for all $n$.
\end{remark}

The main result of \cite{CW-Koornwinder} was the following.
\begin{theorem}\label{thm:CW-K}\cite[Theorem 1.1]{CW-Koornwinder}
The 
Koornwinder moment
$K_{(N-r,0,0,\dots,0)}(\xi)$ (where there are precisely $r$ $0$'s in the partition)
is proportional to the \emph{fugacity partition function}
$\mathcal{Z}_{N,r}(\xi)$ for
the two-species ASEP on a lattice of $N$ sites with $r$ light particles.
More specifically,
$$K_{(N-r,0,0,\dots,0)}(\xi) = \frac{1}{(1-q)^r} \mathcal{Z}_{N,r}(\xi).$$
\end{theorem}

We now need to relate $\mathcal{Z}_{N,r}(\xi)$ to
$\mathbf{Z}_{N,r}(\xi)$, the generating function for rhombic 
staircase tableaux.

\begin{lemma}\label{lem:2partition}
Recall that  
by Theorem \ref{thm:comb-ansatz}, 
the matrices $D^{(t)}, E^{(t)}, A$, and vectors $\langle W|$ and $|V\rangle$ 
satisfy relations (I) through (V) of 
Theorem \ref{thm:MA}, with $\lambda_n = \alpha \beta - q^{n-1} \gamma \delta$.
Let $X$ be an $A$-compatible word in $D^{(t)}, E^{(t)}$ and $A$, 
with $|X|=N$ and $|X|_A = r$.
Let $\mathcal{X}$ be the corresponding word in 
$\DD, \EE,$ and $\AAA$.  For example,
if $X = D^{(0)} A A E^{(2)} A D^{(3)}$, then 
$\mathcal{X} = DAAEAD$.
Then 
$$\frac{\langle \WW | \mathcal{X} | \VV \rangle}{\langle \WW | \AAA^r | \VV\rangle}   = 
\frac{\langle W | X | V \rangle}{\langle W|A^r |V \rangle}
\prod_{i=0}^{N-r-1} (\alpha\beta - q^{i+2r} \gamma \delta)^{-1}. $$
It follows that 
 $$\mathcal{Z}_{N,r}(\xi) = 
\mathbf{Z}_{N,r} (\xi)
\prod_{i=0}^{N-r-1} (\alpha\beta - q^{i+2r} \gamma \delta)^{-1}. $$
\end{lemma}

\begin{proof}
We will prove the first statement of the lemma by induction on $N-r$, 
by comparing relations \eqref{eq:m1} through \eqref{eq:m5} to relations (I) through (V).  
Note that the base case, when $N-r=0$, is trivial.

Now suppose that the lemma holds for all 
$A$-compatible word $X$, where $|X|-|X|_A =k-1$ for some fixed positive $k$.
We want to show that it holds for words $X'$ where 
$|X'|-|X'|_A = k$.

Set $N=|X'|$ and $r=|X'|_A$.
By Theorem \ref{thm:MA}, both $\langle W|X'|V \rangle$ and 
$\langle \WW |\mathcal{X'} | \VV \rangle$ compute 
(unnormalized) steady state probabilities of being in the same
state $\tau$, where $X' = X(\tau)$.  
Therefore 
$\langle W|X'|V \rangle = c_{N,r}
\langle \WW |\mathcal{X'} | \VV \rangle$ for some constant 
$c_{N,r}$ that depends on $N$ and $r$ but not on $X'$.
We want to show that $c_{N,r} = 
\frac{\langle W|A^r|V \rangle}{\langle \WW|\AAA^r|\VV\rangle}
\prod_{i=0}^{N-r-1} (\alpha\beta - q^{i+2r} \gamma \delta)$.

Taking $X'=E^{(0)}X$ and $D^{(0)}X$, respectively, we have that 
\begin{align*}
\alpha \langle W| E^{(0)} X | V\rangle - 
 \gamma \langle W|D^{(0)} X |V\rangle &= 
  c_{N,r} 
\alpha \langle \WW| \EE \mathcal{X} | \VV\rangle - 
 c_{N,r} \gamma \langle \WW| \DD \mathcal{X} |\VV\rangle \\
&= c_{N,r} \langle \WW|\mathcal{X} | \VV \rangle.
\end{align*}

But also by (III), we have that 
$$\alpha \langle W| E^{(0)} X|V \rangle - 
\gamma \langle W|D^{(0)} X|V \rangle = 
\lambda_{N+r} \langle W | X | V \rangle = 
(\alpha \beta - q^{N+r-1} \gamma \delta) \langle W | X | V \rangle.
$$

Therefore we have that 
$$c_{N,r} = \frac{(\alpha \beta - q^{N+r-1} \gamma \delta) 
\langle W|X|V\rangle}{\langle \WW | \mathcal{X}|\VV\rangle},$$
which by the induction hypothesis is equal to 
$\frac{\langle W|A^r|V \rangle}{\langle \WW|\AAA^r|\VV\rangle}
\prod_{i=0}^{N-r-1} (\alpha\beta - q^{i+2r} \gamma \delta)$, as desired.

To prove the second statement of the lemma, note that 
by Corollary \ref{cor:ZNr}, we have that 
$\mathbf{Z}_{N,r}(\xi) = \sum_X \xi^{|X|_D} \langle W| X|V \rangle,$
where the sum is over all $A$-compatible words $X$
such that $|X|=N$ and $|X|_A = r$.
Also note that $\langle W|A^r|V \rangle=1$, since 
there is only one degenerate tableau
of size $(r,r)$, and it has weight $1$.
The second statement of the lemma now follows from the first.
\end{proof}

\bibliographystyle{alpha}
\bibliography{bibliography}

\end{document}